\documentclass{article}

    \PassOptionsToPackage{numbers, compress}{natbib}
\usepackage[final]{neurips_2022}

\usepackage[utf8]{inputenc} % allow utf-8 input
\usepackage[T1]{fontenc}    % use 8-bit T1 fonts
\usepackage{url}            % simple URL typesetting
\usepackage{booktabs}       % professional-quality tables
\usepackage{amsfonts}       % blackboard math symbols
\usepackage{nicefrac}       % compact symbols for 1/2, etc.
\usepackage{microtype}      % microtypography
\usepackage{xcolor}         % colors
\usepackage{amsmath,amsthm,amssymb,bm,bbm}
\usepackage{algorithmic}
\usepackage[ruled,vlined]{algorithm2e}
\usepackage{accents}
\usepackage{graphicx}
\usepackage{babel,blindtext}

\newcommand{\bmat}[1]{\begin{bmatrix}#1\end{bmatrix}}
\newcommand{\tp}{\mathsf{T}}
\newcommand{\norm}[1]{\lVert{#1}\rVert}
\usepackage[bookmarks=false, hidelinks]{hyperref}
\newcommand{\field}[1]{\mathbb{#1}}
\newcommand{\R}{\field{R}}

\DeclareMathOperator*{\tr}{tr}

\DeclareMathOperator*{\diag}{diag}
\DeclareMathOperator*{\conv}{conv}
\DeclareMathOperator*{\dist}{dist}
\DeclareMathOperator*{\lmi}{LMI}
\DeclareMathOperator*{\dom}{dom}

\newtheorem{thm}{Theorem}
\newtheorem{lem}{Lemma}
\newtheorem{prop}{Proposition}
\newtheorem{rem}{Remark}
\newtheorem{defn}{Definition}
\newtheorem{cor}{Corollary}

\newtheorem{assumption}{Assumption}

\title{Global Convergence of Direct Policy Search for State-Feedback $\mathcal{H}_\infty$ Robust Control: A Revisit of Nonsmooth Synthesis with Goldstein Subdifferential}

\author{%
  Xingang Guo, \,\,\,\,\,\,\,\,Bin ~Hu\\
  Department of Electrical and Computer Engineering\\
  Coordinated Science Laboratory\\
  University of Illinois at Urbana-Champaign\\
     \texttt{\{xingang2,binhu7\}@illinois.edu}
}

\begin{document}

\maketitle

\begin{abstract}
Direct policy search has been widely applied in modern reinforcement learning and continuous control. However, the theoretical properties of direct policy search on nonsmooth robust control synthesis have not been fully understood. The optimal $\mathcal{H}_\infty$ control framework aims at designing a policy to minimize the closed-loop $\mathcal{H}_\infty$ norm, and is arguably the most fundamental robust control paradigm. In this work, we show that direct policy search is guaranteed to find the global solution of the robust $\mathcal{H}_\infty$ state-feedback control design problem. Notice that policy search for optimal $\mathcal{H}_\infty$ control leads to a constrained nonconvex nonsmooth optimization problem, where the nonconvex feasible set consists of all the policies stabilizing the closed-loop dynamics. We show that for this nonsmooth optimization problem, all Clarke stationary points are global minimum. Next, we identify the coerciveness of the closed-loop $\mathcal{H}_\infty$ objective function, and prove that all the sublevel sets of the resultant policy search problem are compact. Based on these properties, we show that Goldstein's subgradient method and its implementable variants can be guaranteed to stay in the nonconvex feasible set and eventually find the global optimal solution of the $\mathcal{H}_\infty$ state-feedback synthesis problem. Our work builds a new connection between nonconvex nonsmooth optimization theory and robust control, leading to an interesting global convergence result for direct policy search on optimal $\mathcal{H}_\infty$ synthesis.
\end{abstract}

\section{Introduction}

Reinforcement learning (RL) has achieved impressive performance on many continuous control tasks~\cite{schulman2015high,lillicrap2015continuous}, and
policy optimization is one of the main workhorses for such applications \cite{duan2016benchmarking,sutton2000policy,schulman2015trust,schulman2017proximal}. 
Recently, there have been extensive research efforts studying the global convergence properties of policy optimization methods on benchmark control problems including linear quadratic regulator (LQR)~\cite{pmlr-v80-fazel18a,bu2019lqr,malik2019derivative,yang2019provably,mohammadi2021convergence,furieri2020learning,hambly2021policy}, stabilization \cite{perdomo2021stabilizing,ozaslan2022computing}, linear robust/risk-sensitive control~\cite{zhang2021policy,zhang2020stability,gravell2020learning,zhang2021derivative,zhao2021primal,cui2022mixed}, Markov jump linear quadratic control~\cite{jansch2020convergence,jansch2020policyMDP,jansch2020policy,rathod2021global}, Lur'e system control~\cite{qu2021exploiting}, output feedback control~\cite{fatkhullin2020optimizing,zheng2021analysis,li2021distributed,duan2021optimization,duan2022optimization,mohammadi2021lack,zheng2022escaping}, and dynamic filtering \cite{umenberger2022globally}.
For all these benchmark problems, the objective function in the policy optimization formulation is always differentiable over the entire feasible set, and the existing convergence theory heavily relies on this fact.
Consequently, an important open question remains whether direct policy search can enjoy similar global convergence properties when applied to the famous $\mathcal{H}_\infty$ control problem whose objective function can be non-differentiable over certain points in the policy space \cite{apkarian2006controller,apkarian2006nonsmooth,arzelier2011h2,gumussoy2009multiobjective,burke2020gradient,curtis2017bfgs,noll2005spectral}. 
Different from LQR which considers stochastic disturbance sequences,  $\mathcal{H}_\infty$ control directly addresses the worst-case disturbance, and provides arguably the most fundamental robust control paradigm  \cite{zhou96,Dullerud99,skogestad2007multivariable,basar95,doyle1988state,Gahinet1994}. 
Regarding the connection with RL, it has also been shown that $\mathcal{H}_\infty$ control can be applied to stabilize the training of adversarial RL schemes in the linear quadratic setup \cite[Section 5]{zhang2020stability}.
Given the fundamental importance of $\mathcal{H}_\infty$ control,  we view it as an important benchmark for understanding the theoretical properties of direct policy search in the context of robust control and adversarial RL.  In this work, we study and prove the global convergence properties of direct policy search on the $\mathcal{H}_\infty$ state-feedback synthesis problem.

The objective of the $\mathcal{H}_\infty$ state-feedback synthesis is to design a linear state-feedback policy that stabilizes the closed-loop system and minimizes the $\mathcal{H}_\infty$ norm from the disturbance to a performance signal at the same time.
The design goal is also equivalent to synthesizing a state-feedback policy that minimizes a quadratic cost subject to the worst-case disturbance. We will present the problem formulation for the $\mathcal{H}_\infty$ state-feedback synthesis and discuss such connections in Section~\ref{sec:PF}. Essentially, $\mathcal{H}_\infty$ state-feedback synthesis can be formulated as a constrained policy optimization
problem $\min_{K\in\mathcal{K}} J(K)$, where the decision variable $K$ is a matrix parameterizing the linear state-feedback policy, the objective function $J(K)$ is the closed-loop $\mathcal{H}_\infty$-norm for given $K$, and the feasible set $\mathcal{K}$ consists of all the linear state-feedback policies stabilizing the closed-loop dynamics. Notice that the feasible set for the $\mathcal{H}_\infty$ state-feedback control problem is the same as the nonconvex feasible set for the LQR policy search problem~\cite{pmlr-v80-fazel18a,bu2019lqr}. However, the objective function $J(K)$ for the $\mathcal{H}_\infty$ control problem
can be non-differential over certain feasible points, introducing new difficulty to direct policy search. There has been a large family of nonsmooth $\mathcal{H}_\infty$ policy search algorithms developed based on the concept of Clarke subdifferential \cite{apkarian2006controller,apkarian2006nonsmooth,arzelier2011h2,gumussoy2009multiobjective,burke2020gradient,curtis2017bfgs}. 
However, a satisfying global convergence theory is still missing from the literature. Our paper bridges this gap by making the following two contributions.
\begin{enumerate}
\item We show that all Clarke stationary points for the $\mathcal{H}_\infty$ state-feedback policy search problem are also global minimum. 
\item We identify the coerciveness of the $\mathcal{H}_\infty$ cost function and use this property to show that Goldstein's subgradient method \cite{goldstein1977optimization} and its implementable variants \cite{pmlr-v119-zhang20p, davis2021gradient,burke2020gradient,burke2005robust,kiwiel2007convergence,kiwiel2010nonderivative} can be guaranteed to stay in the nonconvex feasible set of stabilizing policies during the optimization process and eventually find the global optimal solution of the $\mathcal{H}_\infty$ state-feedback control problem. Finite-time complexity bounds for finding $(\delta,\epsilon)$-stationary points are also provided.
\end{enumerate}
Our work sheds new light on the theoretical properties of policy optimization methods on $\mathcal{H}_\infty$ control problems, and serves as a meaningful initial step towards a general global convergence theory of direct policy search on nonsmooth robust control synthesis.

Finally, it is worth clarifying the differences between $\mathcal{H}_\infty$ control and mixed $\mathcal{H}_2/\mathcal{H}_\infty$ design. For mixed $\mathcal{H}_2/\mathcal{H}_\infty$ control, the objective is to design a stabilizing policy that minimizes an $\mathcal{H}_2$ performance bound and satisfies an $\mathcal{H}_\infty$ constraint at the same time \cite{glover1988state,khargonekar1991mixed,kaminer1993mixed,mustafa1991lqg}. In other words, mixed $\mathcal{H}_2/\mathcal{H}_\infty$ control aims at improving the average $\mathcal{H}_2$ performance while ``maintaining" a certain level of robustness by keeping the closed-loop $\mathcal{H}_\infty$ norm to be smaller than a pre-specified number. In contrast, $\mathcal{H}_\infty$ control aims at ``improving" the system robustness and the worst-case performance via achieving the smallest closed-loop $\mathcal{H}_\infty$ norm. In \cite{zhang2021policy}, it has been shown that the natural policy gradient method initialized from a policy satisfying the $\mathcal{H}_\infty$ constraint can be guaranteed to maintain the $\mathcal{H}_\infty$ requirement during the optimization process and eventually converge to the optimal solution of the mixed design problem. However,  notice that the objective function for the mixed $\mathcal{H}_2/\mathcal{H}_\infty$ control problem is still differentiable over all the feasible points, and hence the analysis technique in \cite{zhang2021policy} cannot be applied to our $\mathcal{H}_\infty$ control setting. 
More discussions on the connections and differences between these two problems will be given in the supplementary material.

\section{Problem Formulation and Preliminaries}
\label{sec:PF}
\subsection{Notation}
The set of $p$-dimensional real vectors is denoted as $\R^p$.
For a matrix $A$, we use the notation \(A^\tp\), \( \|A\| \),  \(\tr{A}\), \(\sigma_{\min}(A)\), \(\norm{A}_2\), and \(\rho(A) \) to denote its transpose, largest singular value, trace, smallest singular value, Frobenius norm, and spectral radius, respectively. 
When a
matrix $P$ is negative semidefinite (definite), we will use the notation $P \preceq (\prec) 0$. When $P$ is positive
semidefinite (definite), we use the notation $P \succeq (\succ) 0$.
Consider a (real) sequence $\mathbf{u}:=\{u_0,u_1,\cdots\}$ where
$u_t \in \R^{n_u}$ for all $t$. This sequence is said to be in $\ell_2^{n_u}$
if $ \sum_{t=0}^\infty \| u_t\|^2<\infty$ where $\|u_t\|$ denotes
the standard (vector) 2-norm of $u_t$.  In addition, the $2$-norm for
$\mathbf{u} \in \ell_2^{n_u}$ is defined as
$\|\mathbf{u}\|^2:=\sum_{t=0}^\infty \| u_t\|^2$.

\subsection{Problem statement: $\mathcal{H}_\infty$ state-feedback synthesis and a policy optimization formulation}

We consider the following linear time-invariant (LTI)  system
\begin{align}\label{eq:lti1}
x_{t+1}=Ax_t+Bu_t+w_t, \,\,x_0=0
\end{align}
where $x_t\in\R^{n_x}$ is the state, $u_t\in\R^{n_u}$ is the control action, and $w_t\in\R^{n_w}$ is the disturbance. We have $A\in \R^{n_x\times n_x}$, $B\in \R^{n_x\times n_u}$, and $n_w=n_x$. 
We denote $\mathbf{x}:=\{x_0,x_1,\cdots\}$, $\mathbf{u}:=\{u_0,u_1,\cdots\}$, and $\mathbf{w}:=\{w_0, w_1, \cdots\}$.
The initial condition is fixed as $x_0=0$.
The objective of $\mathcal{H}_\infty$ control is to choose $\{u_t\}$ to minimize the quadratic cost $\sum_{t=0}^\infty (x_t^\tp Q x_t+u_t^\tp R u_t)$ in the presence of the worst-case $\ell_2$ disturbance satisfying $\norm{\mathbf{w}}\le 1$. In this paper, the following assumption is adopted.

\begin{assumption}\label{assump1}
The matrices $Q$ and $R$ are positive definite. The matrix pair $(A,B)$ is stabilizable.
\end{assumption}

 In $\mathcal{H}_\infty$ control, $\{w_t\}$ is considered to be the worst-case disturbance satisfying the $\ell_2$ norm bound $\norm{\mathbf{w}}\le 1$, and can be chosen in an adversarial manner. 
This is different from LQR which makes stochastic assumptions on $\{w_t\}$. 
Without loss of generality, we have chosen the $\ell_2$ upper bound on $\mathbf{w}$ to be $1$. In principle, we can formulate the $\mathcal{H}_\infty$ control problem with any arbitrary $\ell_2$ upper bound on $\mathbf{w}$, and there is no technical difference. We will provide more explanations on this fact in the supplementary material.
Therefore, $\mathcal{H}_\infty$ control can be formulated  as the following minimax problem
\begin{align}\label{eq:minmax}
\min_{\mathbf{u}}\max_{\mathbf{w}:\norm{\mathbf{w}}\le 1} \sum_{t=0}^\infty (x_t^\tp Q x_t+ u_t^\tp R u_t)
\end{align}
Under Assumption \ref{assump1}, it is well known that 
the optimal solution for \eqref{eq:minmax} can be achieved using a linear state-feedback policy $u_t=-Kx_t$ (see \cite{basar95}).
Given any $K$, 
the LTI system \eqref{eq:lti1} can be rewritten as
\begin{align}\label{eq:lti2}
x_{t+1}=(A-BK)x_t+w_t, \,x_0=0.
\end{align}
Now we define $z_t=(Q+K^\tp R K)^{\frac{1}{2}}x_t$. We have $\norm{z_t}^2=x_t^\tp (Q+K^\tp R K) x_t=x_t^\tp Q x_t+u_t^\tp R u_t$. 
We denote $\mathbf{z}:=\{z_0,z_1,\cdots\}$. If $\mathbf{x}\in \ell_2^{n_x}$, then we have $\norm{\mathbf{z}}^2=\sum_{t=0}^\infty (x_t^\tp Q x_t+u_t^\tp R u_t)<+\infty$. 
Therefore, the closed-loop LTI system \eqref{eq:lti2} can be viewed as a linear operator mapping any disturbance sequence $\{w_t\}$ to another sequence $\{z_t\}$. We denote this operator as $G_K$, where the subscript highlights the dependence of this operator on $K$.
If $K$ is stabilizing, i.e. $\rho(A-BK)<1$, then $G_K$ is bounded in the sense that it maps any $\ell_2$ sequence  $\mathbf{w}$ to 
another sequence $\mathbf{z}$ in $\ell_2^{n_x}$. For any stabilizing $K$, the $\ell_2\rightarrow \ell_2$ induced norm of $G_K$ can be defined as:
\begin{align}
\norm{G_K}_{2\rightarrow 2}:=\sup_{0\neq \norm{\mathbf{w}}\le 1}\frac{\norm{\mathbf{z}}}{\norm{\mathbf{w}}}
\end{align}
Since $G_K$ is a linear operator, it is straightforward to show
\begin{align*}
\norm{G_K}_{2\rightarrow 2}^2:=\max_{\mathbf{w}:\norm{\mathbf{w}}\le 1} \sum_{t=0}^\infty x_t^\tp (Q+K^\tp R K) x_t=\max_{\mathbf{w}:\norm{\mathbf{w}}\le 1} \sum_{t=0}^\infty (x_t^\tp Q x_t+u_t^\tp R u_t).
\end{align*}

Therefore, the minimax optimization problem \eqref{eq:minmax} can be rewritten as the policy optimization problem:
$\min_{K\in\mathcal{K}}\norm{G_K}_{2\rightarrow 2}^2$, where $\mathcal{K}$ is the set of all linear state-feedback stabilizing policies, i.e. $\mathcal{K}=\{K\in\R^{n_x\times n_u}: \,\rho(A-BK)<1\}$.  In the robust control literature \cite{apkarian2006controller,apkarian2006nonsmooth,arzelier2011h2,gumussoy2009multiobjective,burke2020gradient,curtis2017bfgs}, it is standard to drop the square in the cost function and just reformulate \eqref{eq:minmax} as $\min_{K\in\mathcal{K}} \norm{G_K}_{2\rightarrow 2}$. This is exactly the policy optimization formulation for $\mathcal{H}_\infty$ state-feedback control.
The main reason why
this problem is termed as $\mathcal{H}_\infty$ state-feedback control is that in the frequency domain, $G_K$ can be viewed as a transfer function which lives in the Hardy $\mathcal{H}_\infty$ space and has an $\mathcal{H}_\infty$ norm being  exactly equal to $\norm{G_K}_{2\rightarrow 2}$. 
Applying the frequency-domain formula for the $\mathcal{H}_\infty$ norm, we can calculate $\norm{G_K}_{2\rightarrow 2}$ as
\begin{align}\label{eq:hinfcost}
\norm{G_K}_{2\rightarrow 2}=\sup_{\omega\in[0, 2\pi]}\lambda_{\max}^{1/2}\big((e^{-j\omega}I-A+BK)^{-\tp}(Q+K^{\tp}RK)(e^{j\omega}I-A+BK)^{-1}\big),
\end{align}
where $I$ is the identity matrix, and $\lambda_{\max}$ denotes the largest eigenvalue of a given symmetric matrix.
Therefore, eventually the $\mathcal{H}_\infty$ state-feedback control problem can be formulated as 
\begin{align}\label{eq:hinfopt}
\min_{K\in\mathcal{K}} J(K),
\end{align}
where $J(K)$ is equal to the $\mathcal{H}_\infty$ norm specified by \eqref{eq:hinfcost}. Classical $\mathcal{H}_\infty$ control theory typically solves \eqref{eq:hinfopt} via introducing extra Lyapunov variables and
reparameterizing the problem into a higher-dimensional convex domain
over which convex optimization algorithms can be applied~\cite{zhou96,Dullerud99,befb94}.
In this paper, we revisit \eqref{eq:hinfopt} as a benchmark for direct policy search, and discuss how to search the optimal solution of \eqref{eq:hinfopt} in the policy space directly.  Applying direct policy search to address \eqref{eq:hinfopt} leads to a nonconvex nonsmooth optimization problem.
A main technical challenge is that the objective function \eqref{eq:hinfcost} can be non-differentiable over some important feasible points \cite{apkarian2006controller,apkarian2006nonsmooth,arzelier2011h2,gumussoy2009multiobjective,burke2020gradient,curtis2017bfgs}.

\subsection{Direct policy search: A nonsmooth optimization perspective}
Now we briefly review several key facts known for the $\mathcal{H}_\infty$ policy optimization problem \eqref{eq:hinfopt}. 
\begin{prop} \label{prop:1}
The set $\mathcal{K}=\{K: \rho(A-BK)<1\}$ is open. In general, it can be unbounded and nonconvex. The cost function \eqref{eq:hinfcost} is continuous and nonconvex in $K$.
\end{prop}
See \cite{pmlr-v80-fazel18a,bu2019topological} for some related proofs.  We have also included more explanations in the supplementary material. An immediate consequence is that \eqref{eq:hinfopt} becomes a nonconvex optimization problem.
Another important fact is that the objective function \eqref{eq:hinfcost} is also nonsmooth. As a matter of fact, \eqref{eq:hinfcost} is subject to two sources of nonsmoothness.  Based on \eqref{eq:hinfcost}, we can see that the largest eigenvalue for a fixed frequency $\omega$ is nonsmooth, and the optimization step over $\omega\in [0, 2\pi]$ is also nonsmooth. As a matter of fact, the $\mathcal{H}_\infty$ objective function \eqref{eq:hinfcost} can be non-differentiable over important feasible points, e.g. optimal points. 
Fortunately, it is well known\footnote{We cannot find a formal statement of Proposition \ref{prop:regular} in the literature. However, based on our discussion with other researchers who have worked on nonsmooth $\mathcal{H}_\infty$ synthesis for long time, this fact is well known and hence we do not claim any credits in deriving this result. As a matter of fact, although not explicitly stated, the proof of Proposition \ref{prop:regular} is hinted in the last paragraph of \cite[Section III]{apkarian2006nonsmooth} given the facts that the $\mathcal{H}_\infty$ norm is a convex function over the Hardy $\mathcal{H}_\infty$ space (which is a Banach space) and the mapping from $K\in\mathcal{K}$ to the (infinite-dimensional) Hardy $\mathcal{H}_\infty$ space is strictly differentiable.
For completeness, a simple proof of Proposition~\ref{prop:regular} based on Clarke's chain rule \cite{clarke1990optimization} is included in the supplementary material.} that the $\mathcal{H}_\infty$ objective function \eqref{eq:hinfcost} has the following desired property so it is Clarke subdifferentiable. 
\begin{prop}\label{prop:regular}
The $\mathcal{H}_\infty$ objective function \eqref{eq:hinfcost} is  locally Lipschitz and subdifferentially regular over the stabilizing feasible set $\mathcal{K}$.
\end{prop}

Recall that  $J:\mathcal{K}\rightarrow \R$ is locally Lipschitz if for
any bounded $S\subset \mathcal{K}$, there exists a constant $L > 0$ such that $|J(K)-J(K')|\le L \norm{K-K'}_2$ for all $K,K'\in S$. 
Based on Rademacher's theorem, a locally Lipschitz function is differentiable almost everywhere, and the Clarke subdifferential is well defined for all feasible points. Formally, the Clarke subdifferential is defined as 
\begin{align}
\partial_C J(K):=\conv\{\lim_{i\rightarrow \infty}\nabla J(K_i):K_i\rightarrow K,\,K_i\in\dom(\nabla J)\subset \mathcal{K}\}
\end{align}
where $\conv$ denotes the convex hull. Then we know that the Clarke subdifferential for the $\mathcal{H}_\infty$ objective function \eqref{eq:hinfcost} is well defined for all $K\in \mathcal{K}$. We say that $K$ is a Clarke stationary point if  $0\in \partial_C J(K)$. The following fact is also well known.
\begin{prop} \label{pro3}
If $K$ is a local min of $J$, then $0\in\partial_C J(K)$ and $K$ is a Clarke stationary point.
\end{prop}
Under Assumption \ref{assump1}, it is well known that there exists $K^*\in \mathcal{K}$ achieving the minimum of~\eqref{eq:hinfopt}. Since $\mathcal{K}$ is an open set, $K^*$ has to be an interior point of $\mathcal{K}$ and hence $K^*$ has to be a Clarke stationary point.  In Section \ref{sec:land}, we will prove that any Clarke stationary points for \eqref{eq:hinfopt} are actually global minimum.

Now we briefly elaborate on the subdifferentially regular property stated in Proposition \ref{prop:regular}.
For any given direction $d$ (which has the same dimension as $K$), the generalized Clarke directional derivative of $J$ is defined as
\begin{align}
J^{\circ}(K,d):=\lim_{K'\rightarrow K}\sup_{t\searrow 0} \frac{J(K'+td)-J(K')}{t}.
\end{align}
In contrast, the (ordinary) directional derivative is defined as follows (when existing)
\begin{align}
J'(K,d):=\lim_{t\searrow 0} \frac{J(K+td)-J(K)}{t}.
\end{align}

In general, the Clarke directional derivative can be different from the (ordinary) directional derivative. Sometimes
 the ordinary directional derivative may not even exist.
The objective function $J(K)$ is subdifferentially regular if for every $K\in \mathcal{K}$, the ordinary directional
derivative always exists and coincides with the generalized one for every direction, i.e. $J'(K,d)=J^{\circ}(K,d)$.
The most important consequence of the subdifferentially regular property is given as follows.
\begin{cor} \label{cor1}
Suppose $K ^\dag\in\mathcal{K}$ is a  Clarke stationary point for $J$. If $J$ is subdifferentially regular, then the directional derivatives $J'(K^\dag, d)$ are non-negative for all $d$.
\end{cor}
See \cite[Theorem 10.1]{rockafellar2009variational} for related proofs and more discussions. Notice that having non-negative directional derivatives does not mean that the point $K^\dag$ is a local minimum. Nevertheless, the above fact will be used in our main theoretical developments.
Now we briefly summarize two key difficulties in establishing a global convergence theory for direct policy search on the $\mathcal{H}_\infty$ state-feedback control problem \eqref{eq:hinfopt}. First, it is unclear whether the direct policy search method will get stuck at some local minimum. Second, it is challenging to guarantee the direct policy search method to stay in the nonconvex feasible set $\mathcal{K}$ during the optimization process. Since $\mathcal{K}$ is nonconvex, we cannot use a projection step to maintain feasibility. Our main results will address these two issues.

\subsection{Goldstein subdifferential}

Generating a good descent direction for nonsmooth optimization is not trivial. Many nonsmooth optimization algorithms are based on the concept of Goldstein subdifferential \cite{goldstein1977optimization}. Before proceeding to our main result, we briefly review this concept here.

\begin{defn}[Goldstein subdifferential]
Suppose $J$ is locally Lipschitz. Given a point $K\in\mathcal{K}$ and a parameter $\delta>0$, the Goldstein subdifferential of $J$ at $K$ is defined to be the following set 
\begin{align} \label{Gold_sub}
\partial_\delta J(K):=\conv \left\{\cup_{K'\in\mathbb{B}_\delta(K)} \partial_C J(K')\right\},
\end{align}
where $\mathbb{B}_\delta(K)$ denotes the $\delta$-ball around $K$. The above definition implicitly requires $\mathbb{B}_\delta(K)\subset\mathcal{K}$.
\end{defn}
Based on the above definition, one can further define the notion of $(\delta,\epsilon)$-stationarity.
A point $K$ is said to be $(\delta,\epsilon)$-stationary if $\dist(0, \partial_\delta J(K))\le \epsilon$.
 It is well-known that the minimal norm
element of the Goldstein subdifferential generates a good descent direction. This fact is stated as follows.
\begin{prop}[\cite{goldstein1977optimization}]
Let $F$ be the minimal norm element in $\partial_\delta J(K)$. Suppose $K-\alpha F/\norm{F}_2\in \mathcal{K}$ for any $0\le \alpha \le \delta$. 
Then we have 
\begin{align}\label{eq:descent}
J(K-\delta F/\norm{F}_2)\le J(K)-\delta \norm{F}_2.
\end{align}
\end{prop}
The idea of Goldstein subdifferential has been used in designing algorithms for nonsmooth $\mathcal{H}_\infty$ control \cite{arzelier2011h2,gumussoy2009multiobjective,burke2020gradient,curtis2017bfgs}. We will show that such policy search algorithms  can be guaranteed to find the global minimum of \eqref{eq:hinfopt}. It is worth mentioning that there are other notions of enlarged subdifferential~\cite{apkarian2006nonsmooth} which can lead to good descent directions for nonsmooth $\mathcal{H}_\infty$ synthesis. In this paper, we focus on the notion of Goldstein subdifferential and related policy search algorithms.

\section{Optimization Landscape for $\mathcal{H}_\infty$ State-Feedback Control}
\label{sec:land}

In this section, we investigate the optimization landscape of the $\mathcal{H}_\infty$ state-feedback policy search problem, and show that any Clarke stationary points of \eqref{eq:hinfopt} are also global minimum. We start by showing the coerciveness of the $\mathcal{H}_\infty$ objective function \eqref{eq:hinfcost}.
\begin{lem}\label{lem1}
The $\mathcal{H}_\infty$ objective function $J(K)$ defined by \eqref{eq:hinfcost} is coercive over the set $\mathcal{K}$ in the sense that for any sequence $\{K^l\}_{l=1}^\infty\subset \mathcal{K}$ we have
    $J(K^l) \rightarrow +\infty$,  
if either $\|K^l\|_2 \rightarrow +\infty$, or  $K^l$ converges to an element in the boundary $\partial \mathcal{K}$.
\end{lem}
\begin{proof}
We will only provide a proof sketch here. A detailed proof is presented in the supplementary material. Suppose we have a sequence $\{K^l\}$ satisfying $\norm{K^l}_2\rightarrow +\infty$. We can choose $\mathbf{w}=\{w_0,0,0,\cdots\}$ with $\norm{w_0}=1$ and show 
$J(K^l)\ge  w_0^\tp (Q+(K^l)^\tp R K^l) w_0 \ge \lambda_{\min}(R) \norm{K^l w_0}^2$.
Clearly, we have used the positive definiteness of $R$ in the above derivation. Then by carefully choosing $w_0$, we can ensure $J(K^l)\rightarrow +\infty$ as $\norm{K^l}_2\rightarrow +\infty$.  Next, we assume $K^l\rightarrow K\in \partial \mathcal{K}$.
We have $\rho(A-BK)=1$, and hence there exists some $\omega_0$ such that $(e^{j\omega_0}I-A+BK)$ becomes singular.
Then we can use the positive definiteness of $Q$ to show
$J(K^l)\ge \lambda^{1/2}_{\min}(Q)   (\|  (e^{j\omega_0}I-A+BK^l)^{-1}   \|\cdot  \|  (e^{-j\omega_0}I-A+BK^l)^{-1}   \|)^{\frac{1}{2}}$.
Notice $\sigma_{\min} (e^{\pm j\omega_0}I-A+BK^l) \to 0$ as $l \to \infty$, which implies $  \|  (e^{\pm j\omega_0}I-A+BK^l)^{-1}   \| \to +\infty$ as $l \to \infty$. Therefore, we have $J(K^l) \to +\infty$ as $K^l\to K\in \partial \mathcal{K}$.
More details for the proof can be found in the supplementary material.
\end{proof}
We want to emphasize that the positive definiteness of $(Q,R)$ are crucial for proving the coerciveness of the cost function \eqref{eq:hinfcost}.  Built upon Lemma \ref{lem1}, we can obtain the following nice properties of the sublevel sets of \eqref{eq:hinfopt}.

\begin{lem}\label{lem2}
Consider the $\mathcal{H}_\infty$ state-feedback policy search problem \eqref{eq:hinfopt} with the objective function $J(K)$ defined in \eqref{eq:hinfcost}. Under Assumption \ref{assump1}, the sublevel set defined as $\mathcal{K}_\gamma:=\{K\in \mathcal{K}: J(K)\le \gamma\}$ is compact and path-connected for every $\gamma\ge J(K^*)$ where $K^*$ is the global minimum of \eqref{eq:hinfopt}.
\end{lem}
\begin{proof}
The compactness of $\mathcal{K}_\gamma$ directly follows from the continuity and coerciveness of $J(K)$, and is actually a consequence of \cite[Proposition 11.12]{bauschke2011convex}. The path-connectedness of the strict sublevel sets for the continuous-time $\mathcal{H}_\infty$ control problem has been proved in \cite{hu2022connectivity}. We can slightly modify the proof in \cite{hu2022connectivity} to show that the strict sublevel set $\{K\in \mathcal{K}: J(K)<\gamma\}$ is path-connected.  Based on the fact that every non-strict sublevel sets are compact, now we can apply \cite[Theorem 5.2]{martin1982connected} to show $\mathcal{K}_\gamma$ is also path-connected.
An independent proof based on the non-strict version of the bounded real lemma is also provided in the supplementary material.
\end{proof}
The path-connectedness of $\mathcal{K}_\gamma$ for every $\gamma$ actually implies the uniqueness of the minimizing set in a certain strong sense \cite[Sections 2\&3]{martin1982connected}. Due to the space limit, we will defer the discussion on the uniqueness of the minimizing set to the supplementary material. Here, we present a stronger result which is one of the main contributions of our paper.
\begin{thm}\label{thm1}
Consider the $\mathcal{H}_\infty$ state-feedback policy search problem \eqref{eq:hinfopt}. Under Assumption \ref{assump1}, any Clarke stationary point of $J(K)$ is a global minimum.
\end{thm}

A detailed proof is presented in the supplementary material.  Here we provide a proof sketch.  Since $Q$ and $R$ are positive definite, the non-strict version of the bounded real lemma\footnote{The difference between the strict and non-strict versions of the bounded real lemma is quite subtle \cite[Section 2.7.3]{befb94}. For completeness, we will provide more explanations for the non-strict version of the bounded real lemma in the supplementary material.} states that $J(K)\le \gamma$ if and only if there exists a positive definite matrix $P$ such that the following matrix inequality holds
\begin{align}\label{eq:lmi1}
\bmat{(A-BK)^\tp P (A-BK) - P & (A-BK)^\tp P \\ P(A-BK) & P}+\bmat{Q+K^\tp R K & 0 \\ 0 & -\gamma^2 I}\preceq 0.
\end{align}
The above matrix inequality is linear in $P$ but not linear in $K$. A standard trick from the control theory can be combined with the Schur complement lemma to convert the above matrix inequality condition to another condition which is linear in all the decision variables \cite{befb94}. Specifically, there exists a matrix function $\lmi(Y,L,\gamma)$ which is linear in $(Y,L,\gamma)$ such that $\lmi(Y,L,\gamma)\preceq 0$ and $Y\succ 0$ if and only if \eqref{eq:lmi1} is feasible with $K=L Y^{-1}$ and $P=\gamma Y^{-1}\succ 0$. The matrix function $\lmi(Y,L,\gamma)$ involves a larger matrix. Hence we present the analytical formula of $\lmi(Y,L,\gamma)$ in the supplementary material and skip it here. Since $\lmi(Y,L,\gamma)$ is linear in $(Y,L,\gamma)$, we know $\lmi(Y,L,\gamma)\preceq 0$ is just a convex semidefinite programming condition. Based on this convex necessary and sufficient condition for $J(K)\le \gamma$, we can prove the following important lemma. 
\begin{lem}\label{lem3}
For any $K\in\mathcal{K}$ satisfying $J(K)>J^*$, there exists a matrix direction $d\neq 0$ such that $J'(K,d)\le J^*-J(K)<0$, where $J^*=J(K^*)$ and $K^*$ is the global minimum of \eqref{eq:hinfopt}.
\end{lem}
\begin{proof}
Suppose we have $K=LY^{-1}$ where $(Y,L, J(K))$ is a feasible point for the convex regime $\lmi(Y,L,J(K))\preceq 0$.
In addition, we have $K^*=L^* (Y^*)^{-1}$ where $(Y^*,L^*, J(K^*))$ is a point satisfying $\lmi(Y^*,L^*,J(K^*))\preceq 0$.
 Since the LMI condition is convex, the line segment between $(Y,L,J(K))$ and $(Y^*,Q^*,J(K^*))$ is also in this convex set. For any $t>0$, we know $(Y+t \Delta Y,L+t \Delta L, J(K)+t(J(K^*)-J(K)))$ also satisfies 
$\lmi(Y+t\Delta Y, L+t\Delta L, J(K)+t(J(K^*)-J(K)))\preceq 0$, 
where $\Delta L=L^*-L$, and $\Delta Y=Y^*-Y$. 
Therefore, based on the bounded real lemma, we know $J((L+t\Delta L) (Y+t \Delta Y)^{-1})\le J(K)+t(J(K^*)-J(K))$.
Let's choose $d=\Delta L Y^{-1} - LY^{-1}\Delta Y Y^{-1}$. Then we have
\begin{align*}
J'(K,d)\le \lim_{t\searrow 0} \left( \frac{J((L+t\Delta L) (Y+t \Delta Y)^{-1})-J(K)}{t}+o(t)\right)\le J^*-J(K)<0.
\end{align*}
A detailed verification of the above inequality is provided in the supplementary material.  Notice $d\neq 0$. If $\Delta L Y^{-1} -LY^{-1}\Delta Y Y^{-1}=0$, the above argument still works and we reach to the conclusion  $J'(K,0)<0$. But this is impossible since we always have  $J'(K,0)=0$.  Hence we have $d\neq 0$. This completes the proof for this lemma.
\end{proof}

Now we are ready to provide the proof for Theorem \ref{thm1}. Based on Lemma \ref{lem3} and the fact that $J(\cdot)$ is subdifferentially regular, the proof  can be done by contradiction.  Suppose $K^*$ is the global minimum, and $K^\dag$ is a Clarke stationary point. If $K^\dag$ is not a global minimum. Then by Lemma \ref{lem3}, there exists $d \neq 0$ such that $J'(K^\dag,d) < 0$, this contradicts the fact that $J'(K^\dag,d) \ge 0$ for all $d$ by Corollary \ref{cor1}. Therefore, $K^\dag$ has to be the global minimum of \eqref{eq:hinfopt}.

The above proof relies on Lemma \ref{lem3} and the fact that $J$ is subdifferentially regular.
Without using the subdifferentially regular property, Lemma \ref{lem3} itself is not sufficient for proving Theorem \ref{thm1}.
It is also worth mentioning that Lemma \ref{lem3} can be viewed as a modification of the convex parameterization/lifting results in \cite{sun2021learning,umenberger2022globally} for non-differentiable points.

\section{Global Convergence of Direct Policy Search on $\mathcal{H}_\infty$ State-Feedback Control}
\label{sec:main1}

In this section, we first show that Goldstein's subgradient method \cite{goldstein1977optimization} can be guaranteed to stay in the nonconvex feasible regime $\mathcal{K}$ during the optimization process and eventually converge to the global minimum of \eqref{eq:hinfopt}. 
The complexity of finding $(\delta,\epsilon)$-stationary points of \eqref{eq:hinfopt} is also presented.
Then we further discuss the convergence guarantees for various implementable forms of Goldstein's subgradient method.

\subsection{Global convergence and complexity of Goldstein's subgradient Method}
We will investigate the global convergence of Goldstein's subgradient method for direct policy search of the optimal $\mathcal{H}_\infty$ state-feedback policy. Goldstein's subgradient method iterates as follows
\begin{align}\label{eq:gold1}
K^{n+1}=K^n-\delta^n F^n/\norm{F^n}_2,
\end{align}
where $F^n$ is the minimum norm element of the Goldstein subdifferential $\partial_{\delta^n} J(K^n)$.
We assume that an initial stabilizing policy is available, i.e. $K^0\in\mathcal{K}$. The same initial policy assumption has also been made in the global convergence theory for direct policy search on LQR \cite{pmlr-v80-fazel18a}. More recently, some provable guarantees have been obtained for finding such stabilizing policies via direct policy search~\cite{perdomo2021stabilizing,ozaslan2022computing}. Hence such an assumption on the initial policy $K^0$ is reasonable.  
Our global convergence result relies on the fact that there is a strict separation between any sublevel set of \eqref{eq:hinfopt} and the boundary of $\mathcal{K}$. This fact is formalized as follows.
\begin{lem}\label{lem4}
Consider the $\mathcal{H}_\infty$ state-feedback policy search problem \eqref{eq:hinfopt} with the cost function $J(K)$ defined in \eqref{eq:hinfcost}.  Denote the complement of the feasible set $\mathcal{K}$ as $\mathcal{K}^c$. 
Suppose Assumption \ref{assump1} holds and $\gamma\ge J^*$. Then there is a strict separation between the sublevel set $\mathcal{K}_\gamma$ and $\mathcal{K}^c$. In other words, we have $dist(\mathcal{K}_\gamma, \mathcal{K}^c)>0$.
\end{lem}
\begin{proof}
Obviously, the set $\mathcal{K}_\gamma \cap \mathcal{K}^c$ is empty (since we know $\mathcal{K}_\gamma\subset \mathcal{K}$). Based on Lemma \ref{lem2}, we know $\mathcal{K}_\gamma$ is compact. Since $\mathcal{K}$ is open, we know $\mathcal{K}^c$ is closed. Therefore, there is a strict separation between $\mathcal{K}_\gamma$  and $\mathcal{K}^c$, and we have $dist(\mathcal{K}_\gamma, \mathcal{K}^c)>0$.
\end{proof}

Now we are ready to present our main convergence result. 
\begin{thm} \label{thm2}
Consider the $\mathcal{H}_\infty$ state-feedback policy search problem \eqref{eq:hinfopt} with the cost function $J(K)$ defined in \eqref{eq:hinfcost}. Suppose Assumption \ref{assump1} holds, and an initial stabilizing policy is given, i.e. $K^0\in\mathcal{K}$. Denote $\Delta_0:=\dist(\mathcal{K}_{J(K^0)}, \mathcal{K}^c)>0$. Choose $\delta^n=\frac{c\Delta_0}{n+1}$ for all $n$ with $c$ being a fixed number in $(0,1)$.
Then Goldstein's subgradient method \eqref{eq:gold1} is guaranteed to stay in $\mathcal{K}$ for all $n$. In addition,  we have $J(K^n)\rightarrow J^*$ as $n\rightarrow \infty$.
\end{thm}
\begin{proof}
We have $\delta^n\le c\Delta_0< \Delta_0$ for all $n$. Now we use an induction proof to show $K^n\in \mathcal{K}_{J(K^0)}$ for all $n$. For $n= 0$, we know $K^0-c\Delta_0 F^0/\norm{F^0}_2$ has to be within the $\Delta_0$ ball around $K^0$ since we know the norm of $F^0/\norm{F^0}_2$ is exactly equal to $1$. Since $\Delta_0:=\dist(\mathcal{K}_{J(K^0)}, \mathcal{K}^c)>0$, we know $K^0-\delta^0 F^0/\norm{F^0}_2\in \mathcal{K}$. As a matter of fact, we know $\mathbb{B}_{\delta^0}(K^0)$ has to be a subset of $\mathcal{K}$. Hence we can apply \eqref{eq:descent} to show that $K^1$ exists and is also in $\mathcal{K}_{J(K^0)}$. Similarly, we can repeat this argument to show $K^n\in \mathcal{K}_{J(K^0)}$ for all $n$. 
Next, we can apply \eqref{eq:descent} to every step and then sum the inequalities over all $n$. Then the following inequality holds for all $N$:
\begin{align}\label{eq:iterative}
\sum_{n=0}^N \delta_n \norm{F^n}_2 \le J(K^0)-J^*
\end{align}
Since we have $\sum_{n=0}^\infty \delta^n=+\infty$, we know $\liminf_{n\rightarrow\infty} \norm{F^n}_2= 0$. There exists one subsequence $\{i_n\}$ such that $\norm{F^{i_n}}_2\rightarrow 0$. For this subsequence, the resultant policy sequence $\{K^{i_n}\}$ is also bounded (notice that the policy parameter sequence stays in the compact set $\mathcal{K}_{J(K^0)}$ for all $n$) and has a convergent subsequence. We can show that the limit of this subsequence is a Clarke stationary point.
Hence the function value associated with this subsequence converges to $J^*$. Notice that $J(K^n)$ is monotonically decreasing for the entire sequence $\{n\}$. Hence we have $J(K^n)\rightarrow J^*$. 
\end{proof}

We have tried to be brief in giving the above proof. We will present a more detailed proof in the supplementary material. We believe that 
this is the first result showing that direct policy search can be guaranteed to converge to the global optimal solution of the $\mathcal{H}_\infty$ state-feedback control problem. The above result only provides an asymptotic convergence guarantee to ensure $J(K^n)\rightarrow J^*$. One can use a similar argument to establish a finite-time complexity bound for finding the $(\delta,\epsilon)$-stationary points of \eqref{eq:hinfopt}. Such a result is given as follows.

\begin{thm} \label{thm3}
Consider the $\mathcal{H}_\infty$ problem \eqref{eq:hinfopt} with the cost function \eqref{eq:hinfcost}. Suppose Assumption \ref{assump1} holds, and $K^0\in\mathcal{K}$. Denote $\Delta_0:=\dist(\mathcal{K}_{J(K^0)}, \mathcal{K}^c)>0$. For any $\delta<\Delta_0$, we can 
choose $\delta^n=\delta$ for all $n$ to ensure that
Goldstein's subgradient method \eqref{eq:gold1} stays in $\mathcal{K}$ and satisfies the following finite-time complexity bound:
\begin{align}
    \min_{n:0\le n\le N} \norm{F^n}_2\le \frac{J(K^0)-J^*}{(N+1)\delta}
\end{align}
In other words, we have $\min_{0\le n\le N} \norm{F^n}_2\le \epsilon$ after $N=\mathcal{O}\left(\frac{\Delta}{\delta\epsilon}\right)$ where $\Delta:=J(K^0)-J^*$. For any $\delta<\Delta_0$ and $\epsilon>0$, the complexity of finding a $(\delta,\epsilon)$-stationary point is $\mathcal{O}\left(\frac{\Delta}{\delta\epsilon}\right)$.
\end{thm}
\begin{proof}
The above result can be proved using a similar argument from Theorem \ref{thm2}. 
We can use the same induction argument to show $K^n\in\mathcal{K}_{J(K^0)}$ for all $n$, and \eqref{eq:iterative} holds with $\delta^n=\delta$. Then the desired conclusion directly follows.
\end{proof}
The complexity for nonsmooth optimization of Lipschitz functions is quite subtle. While the above result gives a reasonable characterization of the finite-time performance of Goldstein's subgradient method on the $\mathcal{H}_\infty$ state-feedback control problem, it does not quantify how fast $J(K^n)$ converges to $J^*$. Recall that $(\delta,\epsilon)$-stationarity means  $\dist(0, \partial_\delta J(K))\le \epsilon$, while $\epsilon$-stationarity means $\dist(0, \partial_C J(K))\le \epsilon$.
As commented in \cite{shamir2020can,pmlr-v119-zhang20p,davis2021gradient},    $(\delta,\epsilon)$-stationarity does not imply being $\delta$-close to an
$\epsilon$-stationary point of $J$. Importantly, the function value of a $(\delta,\epsilon)$-stationary point can be far from $J^*$ even for small $\delta$ and $\epsilon$. Theorem 5 in \cite{pmlr-v119-zhang20p} shows that
there is no finite time algorithm that can find $\epsilon$-stationary points provably for all Lipschitz functions. It is still possible that one can develop some finite time  bounds for $(J(K^n)-J^*)$ via exploiting other advanced properties of the $\mathcal{H}_\infty$ cost function \eqref{eq:hinfcost}. This is an important future task.

\subsection{Implementable variants and related convergence results}
\label{sec:imple}

In practice, it can be difficult to evaluate the minimum norm element of the Goldstein subdifferential.
  Now we discuss implementable variants of Goldstein's subgradient method and  related guarantees. 

\textbf{Gradient sampling  \cite{burke2020gradient,burke2005robust,kiwiel2007convergence}.} 
The gradient sampling (GS) method 
is the main optimization algorithm used in  the robust control package HIFOO \cite{arzelier2011h2,gumussoy2009multiobjective}.
Suppose we can access a first-order oracle which can evaluate $\nabla J$ for any differentiable points in the feasible set\footnote{When $(A,B)$ is known, one can calculate the $\mathcal{H}_\infty$ gradient at differential points using the chain rule in \cite{apkarian2006nonsmooth}. More explanations can be found in the supplementary material.}. Based on Rademacher's theorem, a locally Lipschitz function is differentiable almost everywhere. Therefore, for any $K^n\in\mathcal{K}$, we can randomly sample policy parameters over $\mathbb{B}_{\delta^n}(K^n)$ and obtain differentiable points with probability one. For all these sampled differentiable points, the Clarke subdifferential at each point is just the gradient. Then the convex hull of these sampled gradients can be used as an approximation for the Goldstein subdifferential $\partial_{\delta^n} K^n$. The minimum norm element from the convex hull of the sampled gradients can be solved via a simple convex quadratic program, and is sufficient for generating a reasonably good descent direction for updating $K^{n+1}$ as long as we sample at least $(n_x n_u+1)$ differentiable points for each $n$ \cite{burke2020gradient}.  In the unconstrained setup, the cluster points of the GS algorithm can be guaranteed to be Clarke stationary \cite{kiwiel2007convergence,burke2020gradient}. Such a result can be combined with Theorem \ref{thm1} and Lemma \ref{lem4} to show the global convergence of the GS method on the $\mathcal{H}_\infty$ state-feedback synthesis problem. The following theorem will be treated formally in the supplementary material.
\begin{thm}[Informal statement]\label{thm4}
Consider the policy optimization problem \eqref{eq:hinfopt} with the $\mathcal{H}_\infty$ cost function defined in \eqref{eq:hinfcost}. Suppose Assumption \ref{assump1} holds, and $K^0\in \mathcal{K}$. The iterations generated from the trust-region version of the GS method (described in \cite[Section 4.2]{kiwiel2007convergence} and restated in the supplementary material) can be guaranteed to stay in $\mathcal{K}$ for all iterations and achieve $J(K^n)\rightarrow J^*$ with probability~one. 
\end{thm}

\textbf{Non-derivative sampling (NS) \cite{kiwiel2010nonderivative}.} The NS method can be viewed as the derivative-free version of the GS algorithm. Suppose we only have the zeroth-order oracle which can evaluate the function value $J(K)$ for $K\in \mathcal{K}$. The main difference between NS and GS is that the NS algorithm relies on estimating the gradient from function values via Gupal’s estimation method. In the unconstrained setting, the cluster points of the NS method can be guaranteed to be Clarke stationary with probability one \cite[Theorem 3.8]{kiwiel2010nonderivative}. We can combine \cite[Theorem 3.8]{kiwiel2010nonderivative} with our results (Theorem \ref{thm1} and Lemma~\ref{lem4}) to prove the global convergence of NS in our setting. A detailed discussion is given in the supplementary material.

\textbf{Model-free implementation of NS.}  When the system model is unknown,  there are various methods available for estimating the $\mathcal{H}_\infty$-norm from data \cite{muller2019gain,muller2017stochastic,rojas2012analyzing,rallo2017data,wahlberg2010non,oomen2014iterative,tu2019minimax,tu2018approximation}.  
Based on our own experiences/tests, the multi-input multi-output (MIMO) power iteration method~\cite{oomen2013iteratively} works quite well as a stochastic zeroth-order oracle for the purpose of implementing NS in the model-free setting.  While the sample complexity for model-free NS is unknown, we will provide some numerical justifications to show that such a model-free implementation closely tracks the convergence behaviors of its model-based counterpart.

\textbf{Interpolated normalized gradient descent (INGD) with finite-time complexity.} No finite-time guarantees for finding $(\delta,\epsilon)$-stationary points have been reported for the GS/NS methods. In \cite{pmlr-v119-zhang20p,davis2021gradient}, the INGD method has been developed as another implementable variant of Goldstein's subgradient method, and is proved to satisfy high-probability finite-time complexity bounds for finding $(\delta,\epsilon)$-stationary points of Lipschitz functions. INGD uses an iterative sampling strategy to generate a descent direction which serves a role similar to the minimal norm element of the Goldstein subdifferential. A first-order oracle for differentiable points is needed for implementing the version of INGD in \cite{davis2021gradient}. 
It has been show \cite{pmlr-v119-zhang20p,davis2021gradient} that for unconstrained nonsmooth optimization of $L$-Lipschitz functions\footnote{We slightly abuse our notation by denoting the Lipschitz constant as $L$. Previously, we have used $L$ to denote a particular matrix used in the LMI formulation for $\mathcal{H}_\infty$ state-feedback synthesis.}, the INGD algorithm can be guaranteed to find the $(\delta,\epsilon)$-stationary point with the high-probability iteration complexity $\mathcal{O}\left(\frac{\Delta L^2}{\epsilon^3 \delta}\log(\frac{\Delta}{p\delta\epsilon})\right)$, where $\Delta:=J(K^0)-J^*$ is the initial function value gap, and $p$ is the failure probability (i.e. the optimization succeeds with the probability $(1-p)$).
We can combine the proofs for
\cite[Theorem 2.6]{davis2021gradient} and  Theorem \ref{thm3} to obtain the following complexity result for our $\mathcal{H}_\infty$ setting. A formal treatment is given in the supplementary material.

\begin{thm}[Informal statement]\label{thm5}
Consider the policy optimization problem \eqref{eq:hinfopt} with the $\mathcal{H}_\infty$ cost function defined in \eqref{eq:hinfcost}. Suppose Assumption \ref{assump1} holds, and the initial policy is stabilizing, i.e. $K^0\in \mathcal{K}$.  Denote $\Delta_0:=\dist(\mathcal{K}_{J(K^0)}, \mathcal{K}^c)>0$, and let $L_0$ be the Lipschitz constant of $J(K)$ over the set $\mathcal{K}_{J(K^0)}$. For any $\delta<\Delta_0$, we can 
choose $\delta^n=\delta$ for all $n$ to ensure that the iterations of
the INGD algorithm stay in $\mathcal{K}$ almost surely, and find a $(\delta,\epsilon)$-stationary point with the high-probability iteration complexity $\mathcal{O}\left(\frac{\Delta L_0^2}{\epsilon^3 \delta}\log(\frac{\Delta}{p\delta\epsilon})\right)$, where $p$ is the failure probability.
\end{thm}

\section{Numerical Simulations}
To support our theory, we provide some numerical simulations in this section. The left plot in Figure~\ref{Simulationrelts} shows that GS, NS, INGD, and model-free NS work well for the following example:
\begin{equation} \label{set_matric}
    A = \bmat{1 &0 &-5 \\ -1 &1 &0\\ 0 &0 &1},\,\, B= \bmat{1 \\ 0 \\ -1},\,\, Q = \bmat{2 &-1 &0 \\ -1 &2 &-1 \\ 0 &-1 &2}, \,\, R = 1.
\end{equation}
For this example, we have $J^* = 7.3475$.  We initialize from $K^0 = \bmat{0.4931 &-0.1368 &-2.2654}$, which satisfies $\rho(A-BK^0) = 0.5756 < 1$. The hyperparameter choices are detailed in the supplementary material. We can see that model-free NS closely tracks the trajectory of NS and works well.
In the middle plot of Figure \ref{Simulationrelts}, we  test the NS method on randomly generated cases. We set $A\in \mathbb{R}^{3\times 3}$ to be $I + \xi$,  where each element of $\xi \in \mathbb{R}^{3 \times 3}$ is sampled uniformly from $[0,1]$. For $B\in \mathbb{R}^{3\times 1}$, each element is uniformly sampled from $[0,1]$. We have $Q = I + \zeta I \in \mathbb{R}^{3\times 3}$ with $\zeta$ uniformly sampled from $[0,0.1]$, and $R \in \mathbb{R}$ uniformly sampled from $[1,1.5]$. For each experiment, the initial condition $K^0 \in \mathbb{R}^{1\times 3}$ is also randomly sampled such that $\rho(A-BK^0) < 1$. The NS method converges globally for all the cases. In the right plot, we focus on the model-free setting for \eqref{set_matric}. We decrease the number of samples used in the $\mathcal{H}_\infty$ estimation and show how this increases the noise in the zeroth-order $\mathcal{H}_\infty$ oracle and worsens the convergence behaviors of the model-free NS method. 
Nevertheless, the model-free NS method tracks its model-based counterpart with enough samples.  More numerical results can be found in the supplementary material. 

\begin{figure}
\minipage{0.33\textwidth}
  \includegraphics[width=\linewidth]{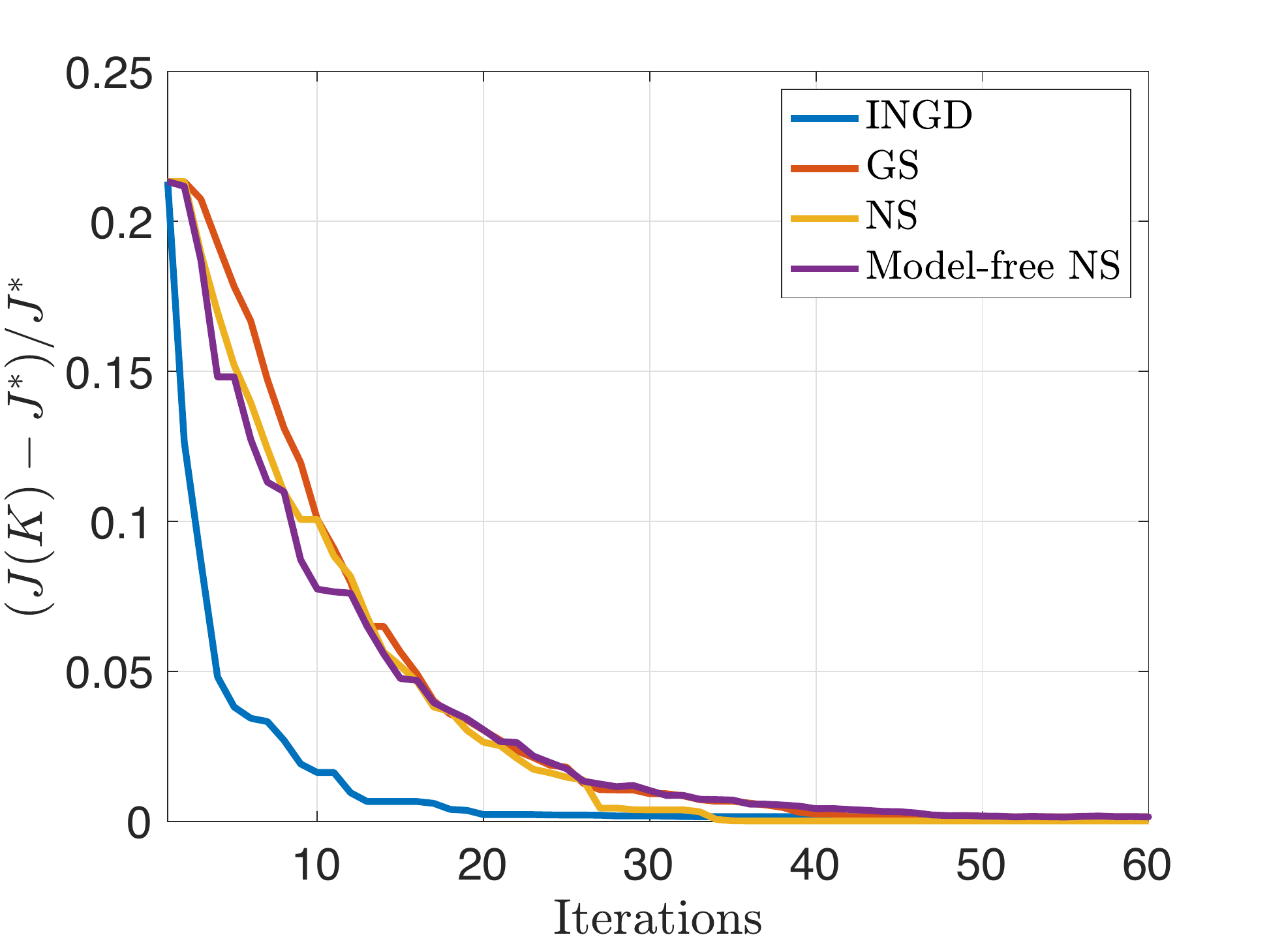}
  \label{fig:awesome_image2}
\endminipage\hfill
\minipage{0.33\textwidth}%
  \includegraphics[width=\linewidth]{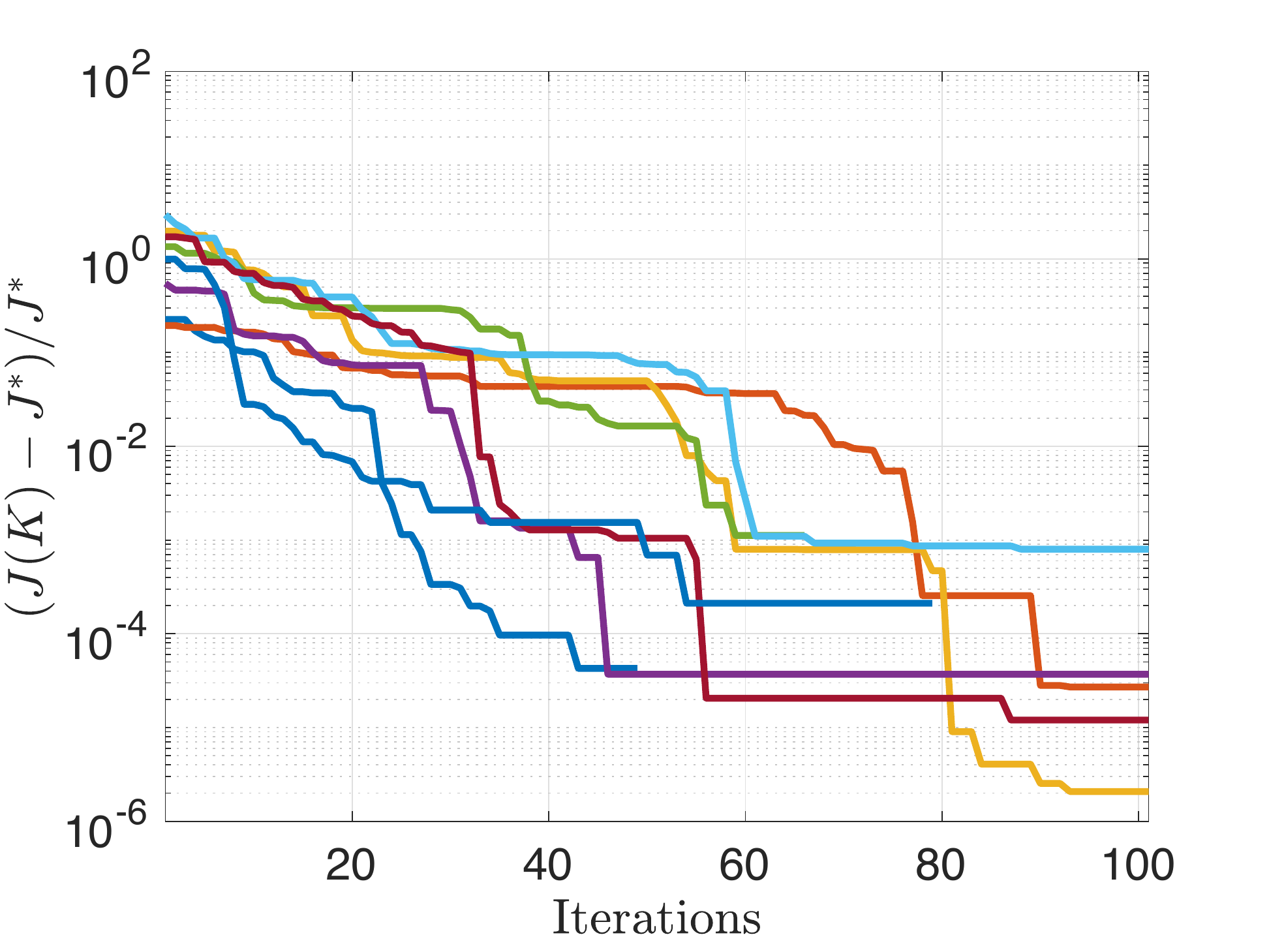}
 \label{fig:awesome_image3}
\endminipage\hfill
\minipage{0.33\textwidth}%
  \includegraphics[width=\linewidth]{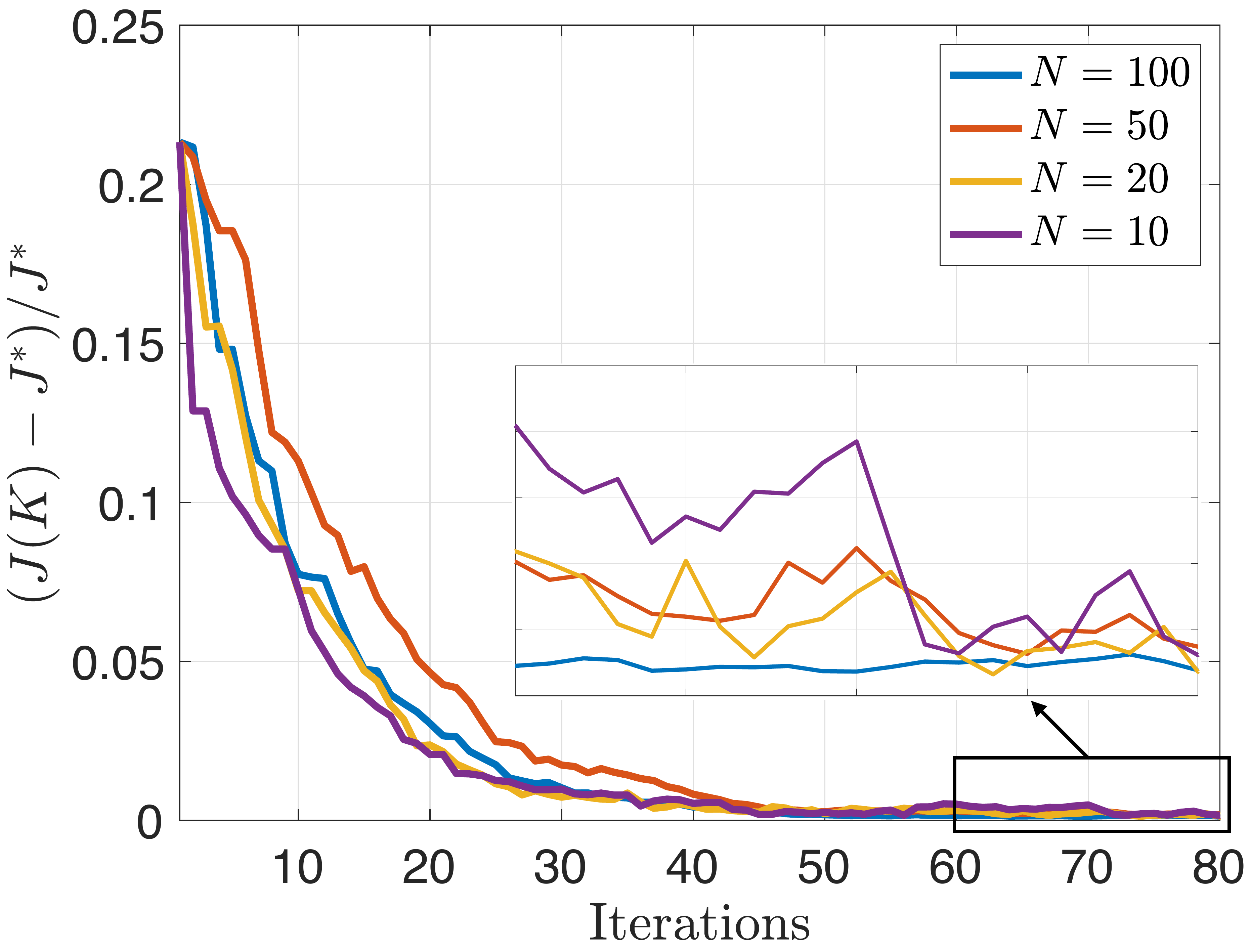}
 \label{fig:awesome_image3}
\endminipage
\caption{Simulation results. Left: The trajectory of relative error of GS, NS, INGD, and Model-free NS methods on \eqref{set_matric}. Middle: The trajectory of relative optimality gap of 8 randomly generated cases for NS method. Right: The trajectory of Model-free NS method with more noisy oracle on \eqref{set_matric}.} \label{Simulationrelts}
\end{figure}

\section{Conclusions and Future Work}

In this paper, we developed the global convergence theory for direct policy search on the $\mathcal{H}_\infty$ state-feedback synthesis problem. Although the resultant policy optimization formulation is nonconvex and nonsmooth, we managed to show that any Clarke stationary points for this problem are actually global minimum, and
the concept of Golstein subdifferential can be used to build direct policy search algorithms which are guaranteed to converge to the global optimal solutions. The finite-time guarantees in this paper are developed only for finding $(\delta,\epsilon)$-stationary points. 
An important future task is to investigate the finite-time bounds for the optimality gap (i.e. $J(K^n)-J^*$) as well as
the sample complexity of direct policy search on model-free $\mathcal{H}_\infty$ control. It is  also of great interests to investigate the convergence properties of direct policy search in nonlinear/output-feedback settings\footnote{Some discussions on possible extensions along this direction have been given in the supplementary material.}.

\begin{ack}
This work is generously supported by the NSF award
CAREER-2048168 and the 2020 Amazon research award.  The authors would like to thank Michael L. Overton, Maryam Fazel, Yang Zheng, Peter Seiler, Geir
Dullerud, Aaron Havens, Darioush Kevian, Kaiqing Zhang, Na Li, Mehran Mesbahi, Tamer Ba\c{s}ar, Mihailo Jovanovic, and Javad Lavaei for the valuable discussions, as well as the helpful suggestions from the anonymous reviewers of NeurIPS.
\end{ack}

\bibliographystyle{abbrv}
{\small
\bibliography{main}
}

%%%%%%%%%%%%%%%%%%%%%%%%%%%%%%%%%%%%%%%%%%%%%%%%%%%%%%%%%%%%

\clearpage
\setcounter{section}{0}
\renewcommand{\thesection}{\Alph{section}}

\begin{center}
{\Large\bf Supplementary Material}
\end{center}
\numberwithin{equation}{section}
\numberwithin{thm}{section}
\numberwithin{lem}{section}
\numberwithin{prop}{section}

\section{More Details on Problem Formulation and Background}

\subsection{Difference between $\mathcal{H}_\infty$ control and mixed $\mathcal{H}_2/\mathcal{H}_\infty$ design}
In this paper, we are interested in $\mathcal{H}_\infty$ control whose objective is to design a state-feedback policy $K$ that stabilizes the closed-loop system and minimizes the closed-loop $\mathcal{H}_\infty$ norm (or equivalently the $\ell_2\rightarrow \ell_2$ induced norm) of $G_K$ at the same time. As discussed in the main paper, $\mathcal{H}_\infty$ control can be formulated as $\min_{K\in \mathcal{K}} \norm{G_K}_{2\rightarrow 2}$.  We have mentioned that such a formulation corresponds to a worst-case assumption on the disturbance $\mathbf{w}$. Since $\norm{G_K}_{2\rightarrow 2}$ is a nonsmooth function in $K$, the $\mathcal{H}_\infty$ control leads to a nonsmooth optimization problem. The global optimal $\mathcal{H}_\infty$ norm for this problem is denoted as $\gamma^*$.

In contrast, the design objective of the mixed $\mathcal{H}_2/\mathcal{H}_\infty$ control is to synthesize a  a linear state-feedback policy which minimizes an upper bound on the $\mathcal{H}_2$ cost and satisfies an additional $\mathcal{H}_\infty$-robustness requirement in the form of $\norm{G_K}_{2\rightarrow 2}<\gamma$ with some pre-specified value of $\gamma$. Recall that we denote the non-strict $\gamma$-admissible sublevel set of the $\mathcal{H}_\infty$ objective function as $\mathcal{K}_\gamma=\{K\in \mathcal{K}: \,\norm{G_K}_{2\rightarrow 2}\le \gamma\}$. Similarly, we definie the strict $\gamma$-admissible sublevel set of the $\mathcal{H}_\infty$ objective function as $\tilde{\mathcal{K}}_\gamma=\{K\in \mathcal{K}: \,\norm{G_K}_{2\rightarrow 2}< \gamma\}$. Then for any $\gamma>\gamma^*$, the $\mathcal{H}_2/\mathcal{H}_\infty$ mixed design problem can be formulated as
\begin{align*}
    \min_{K\in \tilde{\mathcal{K}}_\gamma} \tr(P_K)
\end{align*}
where $P_K$ is the minimal positive definite solution  to the following Riccati equation 
\begin{align*}
(A-BK)^\tp (P_K+P_K(\gamma^2I-P_K)^{-1} P_K) (A-BK)+Q+K^\tp RK-P_K=0.
\end{align*}
Notice that as $\gamma\rightarrow \infty$, the above mixed design problem reduces to the LQR problem. For finite $\gamma$,
$\tr(P_K)$ provides an upper bound for the LQR cost.  For the mixed $\mathcal{H}_2/\mathcal{H}_\infty$ control problem,  the closed-loop $\mathcal{H}_\infty$ norm $\norm{G_K}_{2\rightarrow 2}$ only appears in the constraint. One can show that cost function $\tr(P_K)$ is still differentiable in the policy space for any $K\in \tilde{\mathcal{K}}$. The main result in \cite{zhang2021policy} states that given an initial policy $K^0\in \tilde{\mathcal{K}}_\gamma$ for $\gamma>\gamma^*$, then the natural policy gradient method with well-chosen stepsize is guaranteed to stay in $\tilde{\mathcal{K}}_\gamma$ and fine the optimal policy minimizes $\tr(P_K)$. 
This means that the natural policy gradient method applied to this mixed design problem can improve the average performance (via minimizing $\tr(P_K)$) while maintaining the robustness level $\gamma$.  
A missing step is how to use direct policy search to obtain such an initial policy  satisfying $\norm{G_{K^0}}_{2\rightarrow 2}<\gamma$ in the first place. One by-product of the results in this paper is a direct policy search method for obtaining such initial policy provably. 

It is worth mentioning that it is also possible to to tackle the $\mathcal{H}_\infty$ state-feedback synthesis via solving a sequence of mixed $\mathcal{H}_2/\mathcal{H}_\infty$ design problems, and the approximate central path algorithm in \cite{keivan2021model} is developed based on such an idea.  However, there is no global convergence theory reported for such successive minimization approaches \cite{keivan2021model}.

\subsection{Formulating $\mathcal{H}_\infty$ control with arbitrary $\ell_2$ bound on the disturbance}

To address the worst-case disturbance $\mathbf{w}$,
one may be interested in the following problem formulation
\begin{align}\label{eq:A1}
    \min_{\mathbf{u}} \max_{\mathbf{w}:\norm{\mathbf{w}}\le \Lambda} \sum_{t=0}^\infty (x_t^\tp Q x_t+u_t^\tp R u_t)
\end{align}
where $\Lambda$ is an arbitrary positive number. 
An interesting fact is that \eqref{eq:A1} still leads to the same policy optimization problem $\min_{K\in\mathcal{K}} \norm{G_K}_{2\rightarrow 2}$
regardless of the value of $\Lambda$. 
Now we briefly explain this fact.
Again, it is well-know that it suffices to consider linear state-feedback policies for solving \eqref{eq:A1}. Notice that $G_K$ is a linear operator for any fixed $K$. Therefore, it is straightforward to verify 
\begin{align*}
    \max_{\mathbf{w}:\norm{\mathbf{w}}\le \Lambda} \sum_{t=0}^\infty (x_t^\tp Q x_t+u_t^\tp R u_t)=\Lambda^2 \norm{G_K}_{2\rightarrow 2}^2
\end{align*}
Consequently, \eqref{eq:A1} can be equivalently formulated as $\min_{K\in\mathcal{K}} \Lambda^2 \norm{G}_{2\rightarrow 2}^2$. The constant $\Lambda^2$ can be removed without changing the optimization problem. Therefore, no matter what positive value we use for $\Lambda$, \eqref{eq:A1} is always equivalent to $\min_{K\in\mathcal{K}} \norm{G}_{2\rightarrow 2}$.

\subsection{More explanations for Proposition \ref{prop:1}}

Now we provide some explanations for 
 Proposition \ref{prop:1} which states several facts observed in the existing literature.
The continuity of the cost function \eqref{eq:hinfcost} can be seen from its analytic form. The openness of
$\mathcal{K}$ can be seen from the fact that the spectral radius $\rho(A-BK)$ is a continuous function of $K$.   The unboundedness of $\mathcal{K}$ has been proved in \cite{bu2019topological}. 
To see  that  $\mathcal{K}$ can be  nonconvex, we provide a simple example here. Suppose $A=B=I$, and we choose two policies as follows
\begin{equation*}
    K = \bmat{1 &0\\-1 &1}, \,\,\, K' = \bmat{1 &-5\\0 &1}. 
\end{equation*}
Then one can easily check that $K, K' \in \mathcal{K}$, while their convex combination  $\frac{1}{2}(K_1 + K_2) \notin \mathcal{K}$.

Finally, we provide an example to show $J(K)$ is nonconvex in $K$. Suppose $A=B=Q = R = I$ and choose two policies as follows:
\begin{equation*}
    K = \bmat{0.7 &-3.4\\0.2 &1}, \,\,\, K' = \bmat{0 &-1.3\\0.6 &1.2}. 
\end{equation*}
Let $K'' = \frac{1}{2}(K + K')$, then we have $J(K'') = 35.99$ and $\frac{1}{2}(J(K) + J(K') ) = 13.3$. Hence $J(K'') > 0.5J(K) + 0.5J(K')$, which shows $J(K)$ is nonconvex in $K$.

\subsection{Proof of Proposition \ref{prop:regular}}

Proposition \ref{prop:regular} is a consequence of the chain rule stated in
\cite[Theorem 2.3.10]{clarke1990optimization}. As mentioned previously, Proposition \ref{prop:regular} should be considered as a well-known fact, and we do not claim any credits in proving it.
Although not explicitly stated, the proof of Proposition \ref{prop:regular}
has been hinted in the discussion of \cite[Section III]{apkarian2006nonsmooth}.
 We only present the proof for completeness.

To apply \cite[Theorem 2.3.10]{clarke1990optimization}, we can rewrite the closed-loop $\mathcal{H}_\infty$ norm $\norm{G_K}_{2\rightarrow 2}$ as a composition $g_2\circ g_1(K)$ where $g_1$ maps the policy parameter $K$ to a stable transfer function $(Q+K^\tp R K)^{\frac{1}{2}}(zI-A+BK)^{-1}$ which lives in the infinite-dimensional $\mathcal{H}_\infty$ Hardy space, and $g_2$ maps any stable transfer function in that Hardy space to its $\mathcal{H}_\infty$ norm. It is well known that the infinite-dimensional $\mathcal{H}_\infty$ Hardy space consisting of all stable LTI systems is a Banach space \cite{zhou96}. We also know that  $g_2$ is convex on the Hardy space (by triangle inequality of the $\mathcal{H}_\infty$ norm). The convexity of $g_2$ implies that $g_2$ is also subdifferentially regular.  One can also show that $g_1$ is a strictly differentiable mapping from $K$ to the Hardy space given $\rho(A-BK)<1$ and $Q$ being positive definite. Now we can immediately apply \cite[Theorem 2.3.10]{clarke1990optimization} to obtain the desired conclusion in Proposition \ref{prop:regular}.

\begin{rem} For readability, we also provide some brief explanations for the fact
that the mapping $g_1$ is strictly differentiable. 
For each $K\in \mathcal{K}$, $g_1$ maps $K$ to an infinite-dimensional bounded operator which maps from $\ell_2^{n_w}$ to $\ell_2^{n_x}$ and has a Toeplitz structure. The Toeplitz structure can be combined with Fubini's theorem to show simple upper bounds for the operator norm of $g_1$. One such upper bound is provided by $\norm{(Q+K^\tp R K)^{\frac{1}{2}}}\sum_{k=0}^\infty \norm{(A-BK)^k}$, which is obviously finite for $K\in \mathcal{K}$.
Such a bound allows us to study the operator $g_1$ via treating it as an infinite-dimensional Toeplitz matrix whose properties are completely determined by its first block column.
By definition, $g_1$ admits a strict derivative at $K$ if the following operator norm convergence result holds for any $V$:
\begin{equation} \label{operator_convergence}
    \lim_{K' \to K,\, t \downarrow 0}
    \norm{(g_1(K'+tV)-g_1(K'))/t - \langle D_s g_1(K),V \rangle} = 0,
\end{equation}
where $\norm{\cdot}$ denotes the operator norm, 
and $D_s g_1$ is the candidate for the strict derivative defined via taking derivatives of the infinite-dimensional matrix $g_1$ in a block-by-block manner.
Since we know $\rho(A-BK)<1$, the Toeplitz structure of $g_1$ can be combined with Fubini's theorem again to prove the existence of $D_s g_1$ and provide simple upper bounds for the operator norm on the left side of \eqref{operator_convergence}, leading to the desired operator norm convergence result in \eqref{operator_convergence}. 
\end{rem}

\subsection{A non-strict version of the Bounded Real Lemma}
\label{sec:A5}

The bounded real lemma (also referred to as the Kalman–Yakubovich–Popov (KYP) lemma) is an important tool for characterizing the sublevel sets of the closed-loop $\mathcal{H}_\infty$ norm. There is some subtlety in the assumptions needed for different versions of the bounded real lemma. We briefly clarify this subtlety here.  Consider the following closed-loop linear system:
\begin{align}\label{eq:A2}
\begin{split}
    x_{t+1}&=A_{\mathrm{cl}}x_t+B_{\mathrm{cl}}w_t\\
    z_t&=C_{\mathrm{cl}} x_t,
    \end{split}
\end{align}
where $\rho(A_{\mathrm{cl}})<1$.
Recall that the closed-loop $\mathcal{H}_\infty$ norm of the above system (from $\{w_t\}$ to $\{z_t\}$) is defined as 
\begin{align*}
  H_{cl}= \sup_{\omega\in[0, 2\pi]}\lambda_{\max}^{1/2}\big(B_{\mathrm{cl}}^\tp(e^{-j\omega}I-A_{\mathrm{cl}})^{-\tp}C_{\mathrm{cl}}^\tp C_{\mathrm{cl}}(e^{j\omega}I-A_{\mathrm{cl}})^{-1}B_{\mathrm{cl}}\big).
\end{align*}
The strict version of the bounded real lemma \cite[Theorem 21.12]{zhou96} states that $H_{\mathrm{cl}}<\gamma$ if and only if there exists a positive definite matrix $P$ such that
\begin{align*}
    \bmat{A_{\mathrm{cl}}^\tp P A_{\mathrm{cl}}-P & A_{\mathrm{cl}}^\tp P B_{\mathrm{cl}}\\ B_{\mathrm{cl}}^\tp P A_{\mathrm{cl}} & B_{\mathrm{cl}}^\tp P B_{\mathrm{cl}}}+\bmat{C_{\mathrm{cl}}^\tp C_{\mathrm{cl}} & 0\\ 0 & -\gamma^2 I}\prec 0. 
\end{align*}
The above strict-version version of the bounded real lemma does not require any assumptions on $(A_{\mathrm{cl}},B_{\mathrm{cl}},C_{\mathrm{cl}})$. There is also a non-strict version of the bounded real lemma which requires the extra assumption that $(A_{\mathrm{cl}},B_{\mathrm{cl}},C_{\mathrm{cl}})$ is minimal (i.e. $(A_{\mathrm{cl}},B_{\mathrm{cl}})$ is controllable and $(A_{\mathrm{cl}},C_{\mathrm{cl}})$ is observable). 
Consider \eqref{eq:A2} with $(A_{\mathrm{cl}},B_{\mathrm{cl}},C_{\mathrm{cl}})$ being a minimal realization. Then the non-strict version of the bounded real lemma\footnote{See \cite[Section 2.7.3]{befb94} for the continuous-time counterpart.} states that $H_{\mathrm{cl}}\le\gamma$ if and only if there exists a positive definite matrix $P$ such that
\begin{align*}
    \bmat{A_{\mathrm{cl}}^\tp P A_{\mathrm{cl}}-P & A_{\mathrm{cl}}^\tp P B_{\mathrm{cl}}\\ B_{\mathrm{cl}}^\tp P A_{\mathrm{cl}} & B_{\mathrm{cl}}^\tp P B_{\mathrm{cl}}}+\bmat{C_{\mathrm{cl}}^\tp C_{\mathrm{cl}} & 0\\ 0 & -\gamma^2 I}\preceq 0.
\end{align*}
For the main result in this paper, the proof require the above non-strict version of the bounded real lemma. Specifically, in our setting, we have $\norm{G_K}_{2\rightarrow 2}=H_{\mathrm{cl}}$ if we  choose $A_{\mathrm{cl}}=A-BK$, $B_{\mathrm{cl}}=I$, and $C_{\mathrm{cl}}=(Q+K^\tp R K)^{\frac{1}{2}}$. Since $Q$ is positive definite, such choice of $(A_{\mathrm{cl}},B_{\mathrm{cl}},C_{\mathrm{cl}})$ leads to a minimal realization, and hence it is valid to apply the non-strict version of the bounded real lemma to characterize the non-strict sublevel sets of $\norm{G_K}_{2\rightarrow 2}$.

\section{Detailed Proofs of Our Main Results in Sections \ref{sec:land} \& \ref{sec:main1}}
 In the main paper, we have only provided the proof sketches for Lemma \ref{lem1}, Lemma \ref{lem2},  Theorem \ref{thm1}, Lemma \ref{lem3}, and Theorem \ref{thm2}. In this section, we provide detailed proofs for these results. 

\subsection{Proof for Lemma \ref{lem1}}

Suppose we have a sequence $\{K^l\}$ satisfying $\norm{K^l}_2\rightarrow +\infty$. We can choose $\mathbf{w}=\{w_0,0,0,\cdots\}$ with $\norm{w_0}=1$. Then we have:
\begin{align*}
J(K^l) &= \max_{\mathbf{w}:\norm{\mathbf{w}}\le 1} \sum_{t=0}^\infty x_t^\tp (Q+(K^l)^\tp R K^l) x_t\\
 &\ge  w_0^\tp (Q+(K^l)^\tp R K^l) w_0 \\
 &= w_0^\tp Q w_0 + (K^l w_0)^\tp R ( K^l w_0 )\\
 &\ge \lambda_{\min}(R) \norm{K^l w_0}^2,
\end{align*}
where the first inequality holds since we plugged into a specific $\mathbf{w}$ over the $\max$ operation and the matrix $Q+ (K^l)^{\tp} R K^l$ is positive definite. The second inequality uses the fact that $R \ge \lambda_{\min} (R) I$.  Then by carefully choosing $w_0$, we can ensure $J(K^l)\rightarrow +\infty$ as $\norm{K^l} \rightarrow +\infty$.

Next, we assume $K^l\rightarrow K$ where $K$ is on the boundary $\partial \mathcal{K}$. Clearly we have $\rho(A-BK)=1$. 
We will use a frequency-domain argument to prove $J(K^l)\rightarrow +\infty$.
Since $\rho(A-BK)=1$, there exists some $\omega_0$ such that the matrix $(e^{j\omega_0}I-A+BK)$ becomes singular. Obviously, for the same $\omega_0$, the matrix $(e^{-j\omega_0}I-A+BK)$ is also singular. 
Therefore, we have:
\begin{align*}
J(K^l) &= \sup_{\omega\in[0, 2\pi]}\lambda_{\max}^{1/2}\big((e^{-j\omega}I-A+BK^l)^{-\tp}(Q+(K^l)^{\tp}RK^l)(e^{j\omega}I-A+BK^l)^{-1}\big) \\
&= \sup_{\omega\in[0, 2\pi]} \| (e^{-j\omega}I-A+BK^l)^{-\tp}(Q+(K^l)^{\tp}RK^l)(e^{j\omega}I-A+BK^l)^{-1} \|^{1/2} \\
&\ge \lambda^{1/2}_{\min}(Q) \sup_{\omega\in[0, 2\pi]}  ( \|  (e^{-j\omega}I-A+BK^l)^{-1}   \|\cdot \|  (e^{j\omega}I-A+BK^l)^{-1}   \|)^{1/2} \\
&\ge \lambda^{1/2}_{\min}(Q)   (\| (e^{-j\omega_0}I-A+BK^l)^{-1}   \| \cdot \norm{ (e^{j\omega_0}I-A+BK^l)^{-1}})^{1/2} .
\end{align*}
Clearly, the above argument relies on the fact that $Q$ is positive definite.  
Notice that  we have $\rho(A-BK^l) < 1$ for each $l$ , and hence we have $\sigma_{\min} (e^{\pm j\omega_0}I-A+BK^l)>0$, i.e. the smallest singular values of $(e^{\pm j\omega_0}I-A+BK^l)$ are strictly positive for all $l$. As $l \to \infty$, one can show that the matrices $(e^{\pm j\omega_0}I-A+BK^l)$ converge to the singular matrices $(e^{\pm j\omega_0}I-A+BK)$ with $K=\lim_{l\rightarrow \infty} K^l\in \partial \mathcal{K}$. Hence we have $\sigma_{\min} (e^{\pm j\omega_0}I-A+BK^l) \to 0$ as $l \to \infty$, which implies $  \|  (e^{\pm j\omega_0}I-A+BK^l)^{-1}   \| \to +\infty$ as $l \to \infty$. Therefore, we have $J(K^l) \to +\infty$ as $K^l\to K\in \partial \mathcal{K}$. This completes the proof.

\subsection{Proof for Lemma \ref{lem2}}

We first show the compactness of $\mathcal{K}_\gamma$. Since $J$ is continuous, we know $\mathcal{K}_\gamma = \{ K \in \mathcal{K}: J(K)\le \gamma \}$ is a closed set. It remains to show  $\mathcal{K}_\gamma$ is bounded. Suppose there exist $\gamma > 0$ such that $\mathcal{K}_\gamma$ is unbounded. Then there exists a sequence $\{ K^l \}_{l=1}^\infty \subset \mathcal{K}_\gamma$ such that  $\| K^l \|_2 \to +\infty$ as $l \to \infty$. But by coerciveness of $J(K)$, we must have $J(K^l) \to +\infty$ as well, which contradicts that $J(K^l) \le \gamma$ for all $l$. Hence $\mathcal{K}_\gamma$ is bounded. Therefore, $\mathcal{K}_\gamma$ is compact. The path-connectedness of the strict sublevel sets for the continuous-time $\mathcal{H}_\infty$ control problem has been proved in \cite{hu2022connectivity}. We can slightly modify the proof in \cite{hu2022connectivity} to show that the strict sublevel set $\{K\in \mathcal{K}: J(K)<\gamma\}$ is path-connected.  Based on the fact that every non-strict sublevel sets are compact, we can apply \cite[Theorem 5.2]{martin1982connected} to show $\mathcal{K}_\gamma$ is also path-connected.

We can also prove the path-connectedness of $\mathcal{K}_\gamma$ by using  the non-strict version of the bounded real lemma reviewed in Section \ref{sec:A5}. This proof is more self-contained, and hence also included here.
Since $Q$ and $R$ are positive definite, the non-strict version of the bounded real lemma \cite[Section 2.7.3]{befb94} states that $J(K)\le \gamma$ if and only if there exists a positive definite matrix $P$ such that the following non-strict matrix inequality holds
\begin{align}\label{LMI:mid}
\bmat{(A-BK)^\tp P (A-BK) - P & (A-BK)^\tp P \\ P(A-BK) & P}+\bmat{Q+K^\tp R K & 0 \\ 0 & -\gamma^2 I} \preceq 0.
\end{align}
The above matrix inequality is linear in $P$ but not linear in $K$. A standard trick from the control theory can be combined with the Schur complement lemma to convert the above matrix inequality condition to another condition which is linear in all the decision variable \cite{befb94}. Specifically, there exists a matrix function $\lmi(Y,L,\gamma)$ which is linear in $(Y,L,\gamma)$ such that $Y\succ 0$ and  $\lmi(Y,L,\gamma) \preceq 0$ if and only if \eqref{LMI:mid} is feasible with $K=L Y^{-1}$, and ${P}=\gamma Y^{-1}$. For completeness, we provide the detailed derivation of $\lmi(Y,L,\gamma)$ as follows.

\textbf{Step 1}: Let $\tilde{P} = \frac{1}{\gamma} P $, dividing  both sides of \eqref{LMI:mid} by $\gamma$. Then by Schur complement lemma, \eqref{LMI:mid} can be rewritten as:
\begin{align}\label{supeq:lmi2}
\bmat{(A - BK)^\tp \tilde{P} (A-BK) - \tilde{P} + \frac{1}{\gamma}K^\tp R K  & (A-BK)^\tp \tilde{P} &I \\ \tilde{P}(A-BK) & \tilde{P}-\gamma I &0\\ I &0 &-\gamma Q^{-1}} \preceq 0,
\end{align}

To see this, noticing that $-\gamma Q^{-1}$ is negative definite and \eqref{LMI:mid} can be rewritten as:
\begin{equation}
    \bmat{(A-BK)^\tp \tilde{P} (A-BK)  - \tilde{P} + \frac{1}{\gamma} K^\tp R K & (A-BK)^\tp \tilde{P} \\ \tilde{P}(A-BK) & \tilde{P}-\gamma I} -\bmat{I \\ 0 } (- \frac{1}{\gamma} Q) \bmat{I &0} \preceq 0.
\end{equation}

\textbf{Step 2}:  Now we can apply Schur complement lemma to \eqref{supeq:lmi2} again to show its equivalence to:
\begin{align}\label{supeq:lmi4}
\bmat{ -\tilde{P}+ \frac{1}{\gamma} K^\tp R K  &0 &I &(A-BK)^\tp \tilde{P}  \\ 0 &-\gamma I &0 &\tilde{P}  \\ I &0 &-\gamma Q^{-1} &0 \\
\tilde{P}(A-BK) &\tilde{P} &0 &-\tilde{P} } \preceq 0.
\end{align}

To see this, noticing that $-\tilde{P}$ is negative definite and the LHS of \eqref{supeq:lmi2} can be rewritten as:
\begin{equation}
\bmat{ -\tilde{P} + \frac{1}{\gamma} K^\tp R K  & 0 &I \\ 0 & -\gamma I &0\\ I &0 &-\gamma Q^{-1}} - \bmat{(A-BK)^\tp \tilde{P} \\ \tilde{P} \\ 0} (-\tilde{P}^{-1}) \bmat{\tilde{P}(A-BK) &\tilde{P} &0}.
\end{equation}

\textbf{Step 3}:  Then we left and right multiply $\diag( \bmat{\tilde{P}^{-1} &I &I &\tilde{P}^{-1}})$ on both sides of \eqref{supeq:lmi4}, this will not change the definiteness of \eqref{supeq:lmi4} and leads the LHS of \eqref{supeq:lmi4} to:
\begin{align}
 &\bmat{\tilde{P}^{-1} &0 &0 &0 \\ 0 &I  &0 &0\\ 0 &0 &I &0 \\ 0 &0 &0 &\tilde{P}^{-1}} \bmat{ -\tilde{P}+\frac{1}{\gamma}K^\tp R K  &0 &I &(A-BK)^\tp \tilde{P}  \\ 0 &-\gamma I &0 &\tilde{P}  \\ I &0 &-\gamma Q^{-1} &0 \\
\tilde{P}(A-BK) &\tilde{P} &0 &-\tilde{P} }  \bmat{\tilde{P}^{-1} &0 &0 &0 \\ 0 &I  &0 &0\\ 0 &0 &I &0 \\ 0 &0 &0 &\tilde{P}^{-1}}  \\
 =&  \bmat{ -\tilde{P}^{-1}+ \frac{1}{\gamma}\tilde{P}^{-1}K^\tp R K \tilde{P}^{-1} &0 & \tilde{P}^{-1} &\tilde{P}^{-1} (A-BK)^\tp \\ 0 &-\gamma I &0 &I \\ \tilde{P}^{-1} &0 &-\gamma Q^{-1} &0 \\
 (A-BK)\tilde{P}^{-1} &I &0 &-\tilde{P}^{-1}}  
\end{align}

\textbf{Step 4}:  Substituting $Y =\tilde{P}^{-1}$ and $KY = L$ into the above matrix leads \eqref{supeq:lmi4} to:
\begin{align}\label{supeq:lmi5}
 \bmat{ -Y+ \frac{1}{\gamma}L^\tp R L &0 & Y & (AY-BL)^\tp \\ 0 &-\gamma I &0 &I \\ Y &0 &-\gamma Q^{-1} &0 \\
 AY-BL &I &0 &-Y}    \preceq 0
\end{align}

\textbf{Step 5}:  Furthermore, \eqref{supeq:lmi5} is equivalent to:
\begin{equation} \label{supeq:lmi6}
 \bmat{ -Y &0 & Y & (AY-BL)^\tp \\ 0 &-\gamma I &0 &I \\ Y &0 &-\gamma Q^{-1} &0 \\
 AY-BL &I &0 &-Y} - \bmat{L^\tp \\ 0 \\0 \\0} (-\frac{1}{\gamma}R) \bmat{L &0 &0 &0}   \preceq 0
\end{equation}

Applying Schur complement lemma to \eqref{supeq:lmi6} again leads to:
\begin{equation} \label{eq:bigLMI}
\lmi(Y,L,\gamma):= 
\bmat{ -Y  &0 &Y &(AY-BL)^\tp &L^\tp \\
0 &-\gamma I &0 &I &0 \\
Y &0 &-\gamma Q^{-1} &0 &0 \\
AY-BL &I  &0 &-Y &0\\
L &0 &0 &0 &-\gamma R^{-1}}  \preceq 0.
\end{equation}
 
This proves the equivalence between \eqref{eq:bigLMI} and \eqref{LMI:mid}. Therefore, we have:
 \begin{align}\label{eq:convP}
\begin{split}
     &\{ K \in \mathcal{K}: J(K) \le \gamma \}\\ \Longleftrightarrow\, &\{K: \eqref{LMI:mid}\,\, \text{is feasible},\,P\succ 0\}\\ \Longleftrightarrow\, &\{ K=LY^{-1}: \eqref{eq:bigLMI}\,\, \text{is feasible},\,Y\succ 0\}.
\end{split}
 \end{align}
 
 Noticing that the set of $(Y, L)$ satisfying \eqref{eq:bigLMI} and $Y\succ 0$ is convex and hence path-connected. In addition, the map $K = LY^{-1}$ is continuous for positive definite $Y$.  We can conclude that $\mathcal{K}_\gamma =  \{ K \in \mathcal{K}: J(K) \le \gamma \}$ is path-connected.  
Such a proof is actually quite similar to the proof presented in \cite{hu2022connectivity}. The main difference is that the assumptions in this paper allow us to directly use the non-strict version of the bounded real lemma.

\subsection{Proof for Lemma \ref{lem3}}

The proof of this lemma depends on the above convex parameterization $(Y,L)$.
Since $\lmi(Y,L,\gamma)$ is linear in $(Y,L,\gamma)$, the condition \eqref{eq:bigLMI} is convex. 
Suppose we have $K=LY^{-1}$ where $(Y,L, J(K))$ is a feasible point for the convex regime $\lmi(Y,L,J(K))\preceq 0$ and $Y\succ 0$.
In addition, by the non-strict version of the bounded real lemma, we know there exists a pair $(Y^*,L^*)$ such that $K^*=L^* (Y^*)^{-1}$,  $\lmi(Y^*,L^*,J(K^*))\preceq 0$ and $Y^*\succ 0$.
By convexity of the LMI condition $\lmi(Y,L,\gamma)\preceq 0$ and $Y\succ 0$, we know that the line segment between $(Y,L,J(K))$ and $(Y^*,L^*,J(K^*))$ is also in this convex set. Therefore, for any $0\le t\le 1$, we know the point $(Y+t \Delta Y,L+t \Delta L, J(K)+t(J(K^*)-J(K)))$ also satisfies 
$$\lmi(Y+t\Delta Y, L+t\Delta L, J(K)+t(J(K^*)-J(K)))\preceq 0,\,\, Y+t\Delta Y\succ 0$$
where $\Delta L=L^*-L$, and $\Delta Y=Y^*-Y$. Since $Y\succ 0$ and $Y^*\succ 0$, we automatically have $Y+t\Delta Y\succ 0$. Based on \eqref{eq:convP}, we can construct a policy $(L+t\Delta L) (Y+t\Delta Y)^{-1}$ and the resultant closed-loop $\mathcal{H}_\infty$ norm must be smaller than or equal to $J(K)+t(J(K^*)-J(K))$.
Formally, we have 
$$J((L+t\Delta L) (Y+t \Delta Y)^{-1})\le J(K)+t(J(K^*)-J(K)).$$ 
Based on the fact $J(K^*)<J(K)$, we can use the above inequality to construct a direction $d$ such that $J'(K,d)<0$.
Specifically, let's choose $d=\Delta L Y^{-1} -LY^{-1}\Delta Y Y^{-1}$. Then we have
\begin{align*}
J'(K,d)&= \lim_{t\searrow 0} \frac{J(K + t(\Delta L Y^{-1} -LY^{-1}\Delta Y Y^{-1}))-J(K)}{t}\\
&= \lim_{t\searrow 0} \frac{J(LY^{-1} + t(\Delta L Y^{-1} -LY^{-1}\Delta Y Y^{-1}))-J(K)}{t}\\ 
&=_{(a)} \lim_{t\searrow 0} \frac{J((L+t\Delta L) (Y+t \Delta Y)^{-1}) -J(K)}{t}  \\
&\qquad + \lim_{t\searrow 0}  \frac{ J((L+t\Delta L) (Y+t \Delta Y)^{-1} + O(t^2) ) - J((L+t\Delta L) (Y+t \Delta Y)^{-1}) }{t}\\ 
&\le_{(b)} \lim_{t\searrow 0} \left( \frac{J((L+t\Delta L) (Y+t \Delta Y)^{-1})-J(K)}{t}+O(t)\right)\\
& \le \lim_{t\searrow 0} \left( \frac{J(K)+t(J(K^*)-J(K))-J(K)}{t}+O(t)\right)\\
&= J(K^*) - J(K)  < 0,
\end{align*}
where the step $(a)$ relies on the fact that $(Y + t\Delta Y)^{-1} = Y^{-1} - t Y^{-1} \Delta Y Y^{-1} + O(t^2)$ and the step $(b)$ uses the fact that $J(\cdot)$ is locally Lipschitz (Proposition \ref{prop:regular}).  Finally, the last inequality holds since we know $J(K) > J(K^*)$. Notice $d\neq 0$. Otherwise
the above argument still works and we have $J'(K,0)<0$. This is impossible since we know $J'(K,0)=0$. This leads to a contradiction, and we must have $d\neq 0$. This completes the proof for this lemma.

\subsection{Proof for Theorem \ref{thm2}}
We first show that $K^n\in \mathcal{K}_{J(K^0)}$ for all $n$ by induction. 
By choice, we have $\delta^n = \frac{c \Delta_0}{n+1} \le c\Delta_0$ for all $n$ with some $c \in (0,1)$. For $n=1$, we have $K^1 = K^0-c\Delta_0 F^0/\norm{F^0}_2$. Since the norm of  $F^0/\norm{F^0}_2$ is equal to one and $\delta^0 = c \Delta_0 > 0$, we have  $K^1 \in \mathbb{B}_{\delta^0}(K^0)$, where  $\mathbb{B}_{\delta^0}(K^0)$ is the $\delta^0$-ball centered at $K^0$.  Based on the definition of $\Delta_0$, we know $\mathbb{B}_{\delta^0}(K^0) \subseteq \mathcal{K}$, and hence $K^1 \in \mathcal{K}$. In addition, \eqref{eq:gold1} implies $J(K^1) \le J(K^0) - \delta^0 \| F^0 \|_2$.  Hence we have $\mathcal{K}_{J(K^1)} \subseteq \mathcal{K}_{J(K^0)}$, which implies $K^1 \in \mathcal{K}_{J(K^0)}$. Similarly, we can repeat this argument to show $K^n\in \mathcal{K}_{J(K^0)}$ for all $n$. 

Next, we can apply \eqref{eq:descent} to every step and then sum the inequalities over all $n$. Then the following inequality holds for all $N$:
\begin{align} \label{eq:telosum1}
\sum_{n=0}^N \delta^n \norm{F^n}_2 \le J(K^0)-J(K^{N+1}) \le J(K^0)-J^*,
\end{align}
where the second inequality holds since $J(K^{N+1}) \ge J^*$.
Since we have $\sum_{n=0}^\infty \delta^n = c \Delta_0 \sum_{n=1}^\infty \frac{1}{n} =+\infty$, we know $\liminf_{n\rightarrow\infty} \norm{F^n}_2= 0$. Therefore, there exists one subsequence $\{n_i\}$ such that $\norm{F^{n_i}}_2\rightarrow 0$. For this subsequence, the policy parameter sequence $\{K^{n_i}\}$ stays in the compact set $\mathcal{K}_{J(K^0)}$. Hence we know that the resultant policy sequence $\{K^{n_i}\}$ is also bounded and has a convergent subsequence $\{K^{n_{i(l)}}\}$, which converges to some limit point  $K^\infty\in \mathcal{K}_{J(K^0)}$. In addition, we have $ \| F^{n_{i(l)}} \|_2 \to 0$ as well.

When $\delta > 0$ is sufficiently small, we know $\mathbb{B}_{\delta}(K^\infty)\subset \mathcal{K}$. Then there exists $N_\delta$ such that $\mathbb{B}_{\delta^{n_{i(l)}}}(K^{n_{i(l)}}) \subset \mathbb{B}_{\delta}(K^\infty)$ for all $n_{i(l)} \ge N_\delta$. Hence we have $F^{n_{i(l)}} \in \partial_{\delta^{n_{i(l)}}} J(K^{n_{i(l)}}) \subset \partial_{\delta} J(K^\infty)$ for all $n_{i(l)} \ge N_\delta$. Noticing that $F^{n_{i(l)}} \to 0$ and $\partial_{\cdot} J(\cdot)$ is closed, we have $0 \in \partial_{\delta} J(K^\infty)$ for any $\delta > 0$.
It is also well-know that one has $ \bigcap_{n_{i(l)}} \partial_{\delta^{n_{i(l)}}} J(K^\infty) = \partial_C J(K^\infty)$ (see Remark \ref{rm1} for extra explanations).
Hence we have $0 \in \partial_C (K^\infty)$, 
and $K^\infty$ has to be a Clarke stationary point.
Based on Theorem \ref{thm1}, we have $J(K^\infty)=J^*$, and hence
the function-value subsequence $\{J(K^{n_{i(l)}})\}$ converges to $J^*$. Notice that $\{J(K^n)\}$ is monotonically decreasing for the entire sequence $\{n\}$. Hence the sequence $\{J(K^n)\}$ has a limit, and this limit has to be $J(K^\infty)=J^*$. This completes the proof.

\begin{rem} \label{rm1}
For completeness, 
we also explain the well-known fact regarding $\bigcap_{n} \partial_{\delta^{n}} J(K) = \partial_C J(K)$ as $\delta^n\rightarrow 0$. Here $\{\delta^n\}$ is allowed to be any monotonically-decreasing sequence satisfying $\partial_{\delta^0} J(K)\subset \mathcal{K}$. By definition, we have $\partial_C J(K) \subset \partial_{\delta^n} J(K)$ for all $n$.
Since $\{\delta^n\}$ is monotonically decreasing, we also have $\partial_{\delta^{n+1}} J(K) \subset \partial_{\delta^n} J(K)$ for all $n$. Therefore, we have
    $$\partial_C J(K) \subseteq \lim_{\delta^{n} \searrow 0} \partial_{\delta^{n}} J(K) =  \bigcap_{n} \partial_{\delta^{n}} J(K). $$
To show $\bigcap_{n} \partial_{\delta^{n}} J(K)\subset \partial_C J(K)$,  we can use the following contradiction argument, which is standard (e.g. see \cite[Remark 3.7]{Grohs2016}). Let us assume that there exists $V \in \bigcap_{n} \partial_{\delta^{n}} J(K) \setminus \partial_C J(K)$. Denote $S_n := { \{ \cup_{K'\in \mathbb{B}_{\delta^n}(K)} \partial_C J(K') \} }$. Obviously, $S_n$ depends on $K$ and $\delta^n$.  
In \cite{goldstein1977optimization}, $S_n$ has been shown to be compact and nested. 
By \cite[Lemma 2.6]{goldstein1977optimization}, we have
$$V \in \bigcap_{n} \partial_{\delta^{n}} J(K) = \conv \bigcap_{n}  S_n.$$ 
Therefore, we can express $V$ as $V = \sum_j t_j V_j$ with $V_j \in  \bigcap_{n}  S_n$, $t_j \ge 0$, and $\sum_j t_j = 1$. Notice $V_j \in  S_n$ for all $n$. Based on the definition of Clarke subdifferential, we know that for each $n$,  there exists a sequence of differentiable points $\{ K_j^{n,r}  \}$ such that $K_j^{n,r} \to K_j^n$, $ \nabla J(K_j^{n,r}) \to V_j$ as $r\to \infty$, and $\| K_j^n - K \| \le \delta^n$. Then there exists a large enough $r(n,j)$ such that $\| V_j -  \nabla J(K_j^{n,r(n,j)})\| \le \delta^n $, and $\| K_j^{n,r(n,j)} - K \| \le 2\delta^n$. Since $\delta^n \to 0$ as $n \to \infty$, we have found a sequence $\{ K_j^{n,r(n,j)} \}_{n=1}^\infty$, such that $K_j^{n,r(n,j)} \to K$ and $\nabla J(K_j^{n,r(n,j)}) \to V_j$ as $n \to \infty$. Therefore, we have $V_j \in \partial_C J(K)$.  By convexity of $\partial_C J(K)$ \cite[Proposition 2.3]{goldstein1977optimization}, we know $V = \sum_j t_j V_j \in \partial_C J(K)$. This contradicts the assumption that $V$ is not in $\partial_C J(K)$. Therefore, we must have $\bigcap_{n} \partial_{\delta^{n}} J(K) \subset \partial_C J(K)$. Consequently, we know $\bigcap_{n} \partial_{\delta^{n}} J(K) = \partial_C J(K)$.
\end{rem}

\section{Discussions on Implementable Variants and Related Convergence Results}

In this section, we give more detailed discussions on implementable variants of Goldstein's subgradient method and related convergence guarantees.  Formal treatments to the informal theorems presented in Section \ref{sec:imple} are also presented.

\subsection{Gradient sampling}

\begin{algorithm}
\label{algo4}
\SetAlgoLined
\textbf{Require}: $K^0 \in \mathcal{K}$ at which $J$ is continuously differentiable,  initial sampling radius $\delta^0 \in (0, \frac{\Delta_0}{2})$, initial stationarity target $\epsilon^0 \in [0,\infty)$, reduction factors $\mu_{\delta}, \mu_{\epsilon} \in (0,1]$, sample size $m \ge n_x n_u + 1$, line search parameters $(\beta,\theta) \in (0,1) \times  (0,1)$  \\
\textbf{for}~$n=0,1,2,\cdots$ \textbf{do} \\
\qquad Independently sample $\{ K_{n,1}, \cdots K_{n,m} \}$ uniformly from $\mathbb{B}_{\delta^n} (K^n)$ \\
\qquad Compute $F^n$ as the solution of the following convex quadratic program:
\begin{align*}
    &\min\,\, \frac{1}{2} \| F \|_2^2 \\
    & \text{subject to}\,\, F \in \mathcal{F}^n = \conv\{ \nabla J(K^n), \nabla J(K_{n,1}), \cdots, \nabla J(K_{n,m}) \}
\end{align*} \\ 
\qquad \textbf{if} $\| F^n \|_2 \le \epsilon^n$, set $\delta^{n+1} \leftarrow \mu_{\delta}\delta^n$, $\epsilon^{n+1} \leftarrow \mu_{\epsilon}\epsilon^n$, $t^n \leftarrow 0$, $K^{n+1}\leftarrow K^n$, and move to the next round \\
\qquad \textbf{else} set $\delta^{n+1} \leftarrow \delta^n$, $\epsilon^{n+1}  \leftarrow \epsilon^n$, define:
\begin{equation} \label{GS:normilize}
    \hat{F}^n = \delta^n F^n / \| F^n\|_2,
\end{equation}
and compute $t^n$ such that 
\begin{equation} \label{GS:line_search}
    t^n \leftarrow \max \{ t\in \{1,\theta,\theta^2,\cdots \}: J(K^n-t \hat{F}^n) < J(K^n)-\beta t \delta^n \| F^n \|_2 \}
\end{equation} \\
\qquad \textbf{if} $J$ is  continuously differentiable at $K^n - t^n \hat{F}^n$ \\
\qquad \qquad \textbf{then} set $K^{n+1} \leftarrow K^n - t^n \hat{F}^n$ \\
\qquad \qquad \textbf{else} set $K^{n+1}$ randomly as any point where $J$ is continuously differentiable such that
\begin{align} 
    &J(K^{n+1}) < J(K^n) - \beta t^n \delta^n \| F^n \|_2  \label{GS:random1} \\
    &\| K^{n} - t^n \hat{F}^n - K^{n+1} \|_2 \le \min\{t^n, \delta^n \} \| \hat{F}^n \|_2 \label{GS:random2}
\end{align} \\
\textbf{end for}
\caption{Gradient Sampling (GS)}
\end{algorithm}

The gradient sampling (GS) method can be viewed as an approximated version of Goldstein's subgradient method.
The idea of GS has been explained in Section \ref{sec:imple}. 
In the unconstrained optimization setting, it has been shown that every cluster point of GS can be guaranteed to be Clarke stationary (with probability one) \cite{burke2020gradient,burke2005robust,kiwiel2007convergence}. For our problem, we need to ensure that the iterates do not travel outside the feasible set $\mathcal{K}$, and hence we use the trust-region version of GS, which was originally developed in \cite[Section 4.2]{kiwiel2007convergence}.
For clarity, the trust-region version of GS in \cite[Section 4.2]{kiwiel2007convergence} is restated as Algorithm~\ref{algo4}. For our purpose, we need to ensure $\delta^0<\frac{\Delta_0}{2}$.
See \cite[Section 4.2]{kiwiel2007convergence} for more discussions on the convergence theory for 
Algorithm~\ref{algo4}.
As explained in \cite[Section 4.2]{kiwiel2007convergence}, Theorem 3.3 in \cite{kiwiel2007convergence} also holds for the trust-region version of GS, and hence we know that every cluster point of Algorithm~\ref{algo4} is Clarke stationary (with probability one). From the discussion in \cite[Section 4.2]{kiwiel2007convergence}, we also know that the trust-region version of GS can guarantee $\norm{K^{n+1}-K^n}\le 2t^n \delta^n\le 2\delta^0<\Delta_0$ for any $K^n\in \mathcal{K}$.
Recall that $\Delta_0$ is the distance between $\mathcal{K}_{J(K^0)}$ and $\mathcal{K}^c$.
Therefore, by induction, we can show that the iterations generated by Algorithm \ref{algo4} will be guaranteed to stay in the sublevel set $\mathcal{K}_{J(K^0)}$ with probability one.
Then it is straightforward to  combine the above facts and Theorem~\ref{thm1} to show the global convergence of Algorithm \ref{algo4} for the $\mathcal{H}_\infty$ state-feedback synthesis problem. We formalize Theorem \ref{thm4} as the following global convergence result.

\begin{thm} \label{thm:GS}
Consider the policy optimization problem \eqref{eq:hinfopt} with the $\mathcal{H}_\infty$ cost function defined in \eqref{eq:hinfcost}. Suppose Assumption \ref{assump1} holds, and $K^0\in \mathcal{K}$.
Let $\{ K^n \}$ be a sequence generated by Algorithm \ref{algo4} with  $\mu_{\delta}, \mu_{\epsilon} < 1$. With probability one, we have $K^n\in \mathcal{K}$ for all $n$, and the algorithm does not stop such that $\delta^n \downarrow 0$ and $\epsilon^n \downarrow 0$. In addition, the function-value sequence $\{J(K^n)\}$ is monotonically decreasing almost surely, and we have $J(K^n) \to J^*$ as $n \to \infty$ with probability one. 
\end{thm}
\begin{proof} 
As commented previously, we can use an induction argument to show that the iterates $\{ K^n \}$ generated by Algorithm \ref{algo4} stay inside the feasible set $\mathcal{K}$ almost surely for all $n$.
Now we present some details for this argument.
For $n=0$, we know $\mathbb{B}_{\delta^0}(K^0)\subset \mathcal{K}$, and hence $\{K_{0,1},\cdots, K_{0,m}\}\subset \mathcal{K}$. With probability one, $J$ is differentiable on  $K_{0,j}$   for all $j\in\{1,2,\cdots,m\}$. Therefore, $F^0$ is well defined with probability one.
There are two possibilities. If $\| F^0 \|_2 \le \epsilon^0$, then
we set $K^1 = K^0$ and shrink $(\epsilon^0,\delta^0)$. Obviously, we still have $K^1\in \mathcal{K}_{J(K^0)}$.  If $\| F^0 \|_2 > \epsilon^0$, the algorithm is guaranteed to find a good descent condition satisfying \eqref{GS:line_search}.
Notice that we have $K^0-t \hat{F}^0\in \mathcal{K}$ for any $t\in \{1,\theta,\theta^2,\cdots\}$.
If $J$ is continuously differentiable at $K^0 - t^0 \hat{F}^0$, then we have $K^{1} = K^0 - t^0 \hat{F}^0 \in \mathbb{B}_{\delta^0}(K^0) \subseteq \mathcal{K}$ as $\| \hat{F}^0 \|_2 = \delta^0$ and $t^0 \le 1$. If $J$ is not continuously differentiable at $K^0 - t^0 \hat{F}^0$, then $K^1$ is randomly generated in a way  that \eqref{GS:random1} and \eqref{GS:random2} hold. From \eqref{GS:random2}, we can combine $\delta^0 \in (0, \Delta_0/2)$ and $t^0 \le 1$ to show $\| K^{1} - K^0 \| \le  2 t^0 \delta^0 < \Delta_0$.  Then we have $K^1 \in \mathcal{K}$ as well. Since  the line search guarantees $J(K^1) \le J(K^0)$,  we further have $K^1 \in \mathcal{K}_{J(K^0)}$. 
To summarize, no matter whether $\norm{F^0}_2$ is larger than $\epsilon^0$ or not, we always have 
$K^1 \in \mathcal{K}_{J(K^0)}\subset \mathcal{K}$. 
Then we can repeat this argument to show $K^n \in \mathcal{K}_{J(K^0)}\subset \mathcal{K}$ for all $n$ (with probability one). From \cite[Section 4.2]{kiwiel2007convergence} (and Theorem~\ref{thm1}), every cluster point of $\{ K^n \}$ has to be a Clark stationary point (and hence a global optimal point) of $J$ in the almost sure sense\footnote{As pointed out in \cite{burke2020gradient}, the convergence theory for GS requires the cost function to be continuously differentiable over a set of full measure. Based on the specific form of the $\mathcal{H}_\infty$ cost \eqref{eq:hinfcost} and  the chain rule in \cite{apkarian2006nonsmooth}, one can see that this is not an issue for the $\mathcal{H}_\infty$ state-feedback synthesis problem.}. 
From Lemma~\ref{lem2}, $\mathcal{K}_{J(K^0)}$ is compact.
Then with probability one, $\{K^n \}$ stays in the compact set $\mathcal{K}_{J(K^0)}$ for all $n$, and  there must exist a subsequence of $\{K^n\}$ which converge to one of its cluster points. Therefore,  this subsequence must converge to the global optimal point of $J$, and the function-value sequence associated with this policy subsequence has to converge to $J^*$ almost surely. Notice that $\{J(K^n)\}$ is monotonically decreasing for the entire sequence $\{n\}$ almost surely. Therefore, we have $J(K^n)\rightarrow J^*$ almost surely.
\end{proof}
The implementation of \eqref{GS:random1} and \eqref{GS:random2} can be done via sampling as discussed in \cite[Section 2]{kiwiel2007convergence}. In addition, $\Delta_0$ is typically unknown in practice, and one just needs to tune the parameter $\delta^0$ until it is sufficiently small. If the cost function is not continuously differentiable on the initial stabilizing policy, one can just randomly sample around that stabilizing policy with a sufficiently small sampling radius to obtain another stabilizing policy meeting the initialization requirement in Algorithm \ref{algo4}.

\vspace{-0.1in}
\subsection{Non-derivative sampling}

 The non-derivative sampling (NS) method was originally developed in \cite{kiwiel2010nonderivative}, and can be viewed as the derivative-free version of GS. 
For the NS method, the gradient oracle in the GS algorithm is replaced by the zeroth-order oracle which is only required to return the function value $J(K)$ for given $K$, and similar convergence guarantees can still be obtained.
 To avoid the gradient evaluation,  the Gupal estimation is used in the NS method to approximate the derivatives from function values. For completeness, we restate the NS method from \cite{kiwiel2010nonderivative} as follows. 
 The mapping $\chi$ used in the algorithm description is given below by \eqref{eq:def_chi}, and  $\mathcal{Z}$ denotes the uniform distribution over the $n_u\times n_x$ unit cube, i.e. $[-1/2,1/2]^{n_u\times n_x}$. Any $z\in \mathcal{Z}$ will be an $n_u\times n_x$ matrix.

\begin{algorithm}
\label{algo5}
\SetAlgoLined
\textbf{Require}: initial stabilizing policy $K^0 \in \mathcal{K}$,   initial sampling radius $\delta^0 \in (0, \Delta_0/2)$, initial stationarity target $\epsilon^0 \in [0,\infty)$, reduction factors $\mu_{\delta}, \mu_{\epsilon} \in (0,1]$, sample size $m \ge n_x n_u + 1$, line search parameters $(\beta,\underline{t},\kappa)$ in  $(0,1)$, and a sequence of positive mollifier parameters defined as $\alpha_n = \alpha_0/(n+1)$ with $\alpha_0 < \min\{\Delta_0/\sqrt{n_x n_u},1\}$. \\
\textbf{for}~$n=0,1,2,\cdots$ \textbf{do} \\
\qquad Independently sample $\{ K_{n,1}, \cdots K_{n,m} \}$ uniformly from $\mathbb{B}_{\delta^n} (K^n)$ \\
\qquad Independently sample $\{ z_{n,1}, \cdots z_{n,m} \}$ uniformly from $\mathcal{Z}$ \\
\qquad Compute $F^n$ as the solution of
\begin{align*}
    &\min\,\, \frac{1}{2} \| F \|_2^2 \\
    & \text{subject to}\,\, F \in \mathcal{F}^n = \conv\{ \chi(K_{n,1},\alpha_n,z_{n,1}), \cdots, \chi(K_{n,m},\alpha_n,z_{n,m}) \}
\end{align*} \\ 
\qquad \textbf{if} $\| F^n \| \le \epsilon^n$, set $\epsilon^{n+1}  \leftarrow \mu_{\epsilon} \epsilon^n$, $\delta^{n+1}  \leftarrow \mu_{\delta} \delta^n$, $t^n \leftarrow 0$, $K^{n+1}\leftarrow K^n$, and move to the next round \\
\qquad \textbf{else} set $\delta^{n+1} \leftarrow \delta^n$, $\epsilon^{n+1}  \leftarrow \epsilon^n$, $\hat{F}^n \leftarrow F^n / \| F^n\|_2$, and $K^{n+1}\leftarrow K^n-t^n \hat{F}^n$, where $t^n$ is determined using the following line search strategy: \\
\qquad \qquad (i) Choose an initial step size $t = t_{ini}^n  = \delta^n\ge t^n_{\min} :=  \min\{\underline{t}, \kappa \delta^n/3\}$\\
\qquad \qquad (ii) If $J(K^n - t \hat{F}^n  ) \le J(K^n) - \beta t \| F^n \|$, return $t^n := t$ \\
\qquad \qquad (iii) If $\kappa t < t^n_{\min}$,  return $t^n := 0$ \\
\qquad \qquad (iv) Set $t:= \kappa t$, and go to (ii). \\
\textbf{end for}
\caption{Non-derivative  Sampling (NS)}
\end{algorithm}
For any $z\in \mathcal{Z}$, 
 we denote the $(i,j)$-entry of $z$ as $z(i,j)$. 
For every $K\in \mathcal{K}$ and $z\in \mathcal{Z}$, we formally define $\chi(K,\alpha,z)$ to be a $n_u\times n_x$ matrix\footnote{Obvious, the dimensions of $\chi(K,\alpha,Z)$ and $K$ are the same.} whose $(i,j)$-th entry is calculated as 
\begin{align}\label{eq:def_chi}
    \chi_{ij}(K,\alpha,z)=\frac{1}{\alpha}\left(J(K+\alpha z+V_{ij}^+)-J(K+\alpha z+V_{ij}^-)\right),
\end{align}
where $V_{ij}^+\in \R^{n_u\times n_x}$ is a matrix whose $(i,j)$-th entry is equal to $\frac{\alpha}{2}-\alpha z(i,j)$ and all other entries are $0$, and $V_{ij}^-\in \R^{n_u\times n_x}$ is a matrix whose $(i,j)$-entry is equal to $-\frac{\alpha}{2}-\alpha z(i,j)$  and all other entries are $0$. For our constrained optimization setup,
the above definition assumes $K+\alpha z+V_{ij}^{\pm}\in \mathcal{K}$ for all $(i,j)$, and hence
implicitly requires that $\alpha$ is small enough such that $K+[-\alpha/2,\alpha/2]^{n_u\times n_x}\subset \mathcal{K}$.
Notice that $\chi(K,\alpha,z)$ is exactly the Gupal estimate for the gradient of some smoothed version (or formally the Steklov average) of $J(K)$.  See \cite[Section 2]{kiwiel2010nonderivative} for more detailed discussions. 
Based on the definition of $\chi(K,\alpha,z)$, the NS method can be implemented as described in Algorithm \ref{algo5}.

For Algorithm \ref{algo5}, we choose $\delta^0 \in (0,\Delta_0/2)$ and $\alpha_0< \min\{\Delta_0/\sqrt{n_x n_u},1\}$ to ensure $F^n$ is well defined and $K^n\in\mathcal{K}$ almost surely for all $n$. 
With the help of \cite[Theorem 3.8]{kiwiel2010nonderivative} and Theorem \ref{thm1}, we can  establish the following global convergence result.

\begin{thm} \label{NS_global}
Consider the policy optimization problem \eqref{eq:hinfopt} with the $\mathcal{H}_\infty$ cost function defined in \eqref{eq:hinfcost}. Suppose Assumption \ref{assump1} holds, and $K^0\in \mathcal{K}$. Let $\{ K^n \}$ be a sequence generated by Algorithm \ref{algo5} with  $\mu_{\delta}, \mu_{\epsilon} < 1$. For every step, $F^n$ is well defined. With probability one, $K^n \in \mathcal{K}$ for all $n$, and the algorithm does not stop such that $\delta^n \downarrow 0$ and $\epsilon^n \downarrow 0$. In addition,
the function-value sequence $\{ J(K^n) \}$  is monotonically decreasing almost surely, and we have $J(K^n) \to J^*$ as $n \to \infty$ with probability one. 
\end{thm}
\begin{proof}
We can use an induction argument to show that  $F^n$ is well defined, and $K^n\in \mathcal{K}$ almost surely.
For $n=0$, we know $\mathbb{B}_{\delta^0}(K^0)\subset \mathcal{K}$, and hence $\{K_{0,1},\cdots, K_{0,m}\}\subset \mathcal{K}$. 
To ensure $F^0$ being well defined, we need $ \chi(K_{0,l},\alpha_0,z_{0,l})$ to be well defined for all $l\in \{1,2,\cdots,m\}$. As discussed above, this can be guaranteed via ensuring $K_{0,l} + [-\alpha_0/2,\alpha_0/2]^{n_u\times n_x} \subset \mathcal{K}$ for all $l$. Since we have chosen $\alpha_0 < \Delta_0 / \sqrt{n_x n_u}$ and $\delta^0<\Delta^0/2$, we must have $K_{0,l} + [-\alpha_0/2,\alpha_0/2]^{n_u\times n_x} \subset \mathbb{B}_{\Delta_0/2}(K_{0,l})\subset \mathbb{B}_{\delta^0+\Delta^0/2}(K^0)\subset \mathcal{K}$. Therefore, $F^0$ is well defined. 
If $\| F^0 \|_2 \le \epsilon^0$, then
we set $K^1 = K^0$ and shrink $(\epsilon^0,\delta^0)$. Obviously, we  have $K^1\in \mathcal{K}_{J(K^0)}$.  
If $\| F^0 \|_2 > \epsilon^0$, we need to compute the step size $t^0$ as described in Algorithm \ref{algo5}. There are two possibilities: if $t^0 = 0$, then we still have $K^1 = K^0$ and hence $K^1\in \mathcal{K}_{J(K^0)}$. If $t^0 \neq 0$, then the algorithm has found a good descent condition such that $J(K^0 - t^0 \hat{F}^0  ) \le J(K^0) - \beta t^0 \| F^0 \|$ holds. In this case, we set $K^1 = K^0 - t^0 \hat{F}^0$. 
Notice that we must have $K^0-t^0 \hat{F}^0\in \mathcal{K}$ as $t^0 \le \delta^0$ and $\norm{\hat{F}^0}_2 = 1$. 
Since the descent condition guarantees $J(K^1) \le J(K^0)$,  we further have $K^1 \in \mathcal{K}_{J(K^0)}$. 
To summarize, no matter whether $\norm{F^0}_2$ is larger than $\epsilon^0$ or not, we always have 
$K^1 \in \mathcal{K}_{J(K^0)}\subset \mathcal{K}$. 
Then we can repeat this argument to show $F^n$ is well defined, and  $K^n \in \mathcal{K}_{J(K^0)}\subset \mathcal{K}$ almost surely for all $n$.
From \cite[Theorem 3.8]{kiwiel2010nonderivative} (and Theorem~\ref{thm1}), every cluster point of $\{ K^n \}$ has to be a Clark stationary point (and hence a global optimal point) of $J$ in the almost sure sense.
From Lemma~\ref{lem2}, $\mathcal{K}_{J(K^0)}$ is compact, and $\{K^n\}$  must admit a subsequence which converge to one of its cluster points. Therefore,  with probability one, this subsequence must converge to the global optimal point of $J$, and the function-value sequence associated with this policy subsequence has to converge to $J^*$. Notice that $\{J(K^n)\}$ is bounded and monotonically decreasing for the entire sequence $\{n\}$ in an almost sure sense. Therefore, we have $J(K^n)\rightarrow J^*$ with probability one.
\end{proof}

Notice that $\Delta_0$ is typically unknown when implementing Algorithm \ref{algo5}. Hence one just tunes the values of $\delta^0$ and $\alpha^0$ until they are sufficiently small.

\subsection{Model-free NS}
When the model is unknown,  one can estimate the value of $J(K)$ from sampled trajectories of the closed-loop system, and implement a model-free version of NS via using the resultant stochastic zeroth-order oracle. 
As a matter of fact, there are many different methods available for estimating the $\mathcal{H}_\infty$-norm from data \cite{muller2019gain,muller2017stochastic,rojas2012analyzing,rallo2017data,wahlberg2010non,oomen2014iterative,tu2019minimax,tu2018approximation}.  
For the model-free NS method, the exact function-value oracle for the $\mathcal{H}_\infty$ cost \eqref{eq:hinfcost} is replaced by a stochastic oracle which relies on noisy estimates of the cost value.
Based on our experience, the multi-input multi-output (MIMO) power iteration method~\cite{oomen2013iteratively} works quite well as a stochastic zeroth-order oracle for the purpose of implementing the model-free NS method.
The main idea of the MIMO power iteration method 
 is that the $\mathcal{H}_{\infty}$ norm of an LTI  MIMO system can be estimated from the largest singular value of its finite-time approximated representation (which can be thought as a matrix), and a specialized time reversal trick can be used to make the computation efficient. Given a black-box simulator for a stable system $G_K$,  the MIMO power iteration method provides a reasonably good oracle for estimating the $\mathcal{H}_\infty$ norm of $G_K$ from sampled trajectories. It is worth mentioning that the estimation quality depends on the length of the time window over which $G_K$ is approximated. 
 We denote this window length as $N$.
 The larger $N$ is, the better the $\mathcal{H}_\infty$-norm estimation is. 
We refer the readers to ~\cite{oomen2013iteratively} for implementation details of the MIMO power iteration method.

 The sample complexity for model-free NS is still  unknown. However, our numerical results show that such  a model-free method tracks the convergence of its model-based counterpart given reasonable amount of data\footnote{Specifically, $N$ should be chosen properly and cannot be too small.}.

\subsection{Interpolated Normalized Gradient Descent (INGD) with finite-time complexity }
 \label{INGD1}
 Both GS and NS methods do not have finite-time complexity guarantees for finding $(\delta,\epsilon)$-stationary points. 
The INGD method was originally proposed 
in \cite{pmlr-v119-zhang20p}, and provide an alternative implementable approximation for Goldstein's subgradient method. The main advantage of INGD is that it yields finite-time iteration complexity for finding $( \delta,\epsilon)$-stationary points. The original version of INGD developed in \cite{pmlr-v119-zhang20p} requires a special generalized gradient oracle (see \cite[Assumption 1]{pmlr-v119-zhang20p} for details).
To relax this requirement, \cite{davis2021gradient} proposed a variant of  INGD  which only requires standard gradient oracle for any differentiable points. Specially, at an iterate $K^n$, the INGD method in \cite{davis2021gradient} relies on Algorithm \ref{algo3} to compute the update direction $F^n$.
Again, we slightly abuse our notation by using $L$ to denote the Lipschitz constant appearing in Algorithm \ref{algo3} (previously, we used $L$ to denote a particular matrix in the LMI formulation for $\mathcal{H}_\infty$ state-feedback synthesis). 
In the unconstrained setup, it has been shown that Algorithm \ref{algo3} terminates and generates a good descent direction with high probability \cite[Corollary 2.5]{davis2021gradient}. Then one can combine Algorithm~\ref{algo3} and \eqref{eq:gold1} to formulate the INGD method, which is formally given as Algorithm \ref{algo2}. In the unconstrained optimization setting,  finite-time iteration complexity for finding $(\delta,\epsilon)$-stationary points have been obtained for Algorithm \ref{algo2} \cite[Theorem 2.6]{davis2021gradient}.

\begin{algorithm}
\label{algo3}
\SetAlgoLined
\textbf{Input}: $K$, $\delta > 0$, $\epsilon > 0$, and the Lipschitz constant $L$  \\
Set $F = \nabla J(\Xi)$ where $\Xi$ is sampled uniformly from $\mathbb{B}_\delta (K)$  \\
\textbf{while}~$\| F \|_2 > \epsilon$ \textbf{and} $ \frac{\delta}{4} \| F \|_2 \ge J(K) - J(K - \delta F/{\|F \|_2})$ ~\textbf{do} \\
\qquad Choose any $r$ satisfying $0 < r < \| F \|_2 \sqrt{1- (1- \frac{\|F \|^2_2}{128 L^2})^2}$ \\
\qquad Sample $\Upsilon$ uniformly from $\mathbb{B}_r(F)$ \\
\qquad Sample $\Xi$ uniformly from $[K,K - \delta \frac{\Upsilon}{\|\Upsilon \|_2}]$ \\
\qquad $F \leftarrow \arg \min_{\Phi \in [F, \nabla J(\Xi)]} \| \Phi \|_2$ \\
\textbf{end while} \\
Return $F$
\caption{MinNorm}
\end{algorithm}

\begin{algorithm}
\label{algo2}
\SetAlgoLined
\textbf{Initial}: $K^0$, $T$  \\
\textbf{Input}: $\delta$, $\epsilon$, and the Lipschitz constant $L$  \\
\textbf{for}~$n = 0,1,\cdots, T$ \textbf{do} \\
\qquad $F^n = \text{MinNorm}(K^n,\delta,\epsilon,L)$ \\
\qquad $K^{n+1} = K^n - \delta F^n / \|F^n \|_2$ \\
\textbf{end} \\
Return $ K^T $
\caption{ Interpolated Normalized Gradient Descent (INGD)}
\end{algorithm}

Now we discuss how to modify the finite-time analysis in \cite{davis2021gradient} for our $\mathcal{H}_\infty$ control problem. It has been shown in \cite[Theorem 2.6]{davis2021gradient} that for unconstrained nonsmooth optimization of $L$-Lipschitz functions, the above INGD algorithm can be guaranteed to find one $(\delta,\epsilon)$-stationary point with the high-probability iteration complexity $\mathcal{O}\left(\frac{\Delta L^2}{\epsilon^3 \delta}\log(\frac{\Delta}{p\delta\epsilon})\right)$, where $\Delta$ is the initial function value gap, and $p$ is the failure probability (i.e. the optimization succeeds with the probability $(1-p)$).  We can choose $\delta<\Delta_0$ to 
ensure that the iterates from Algorithm \ref{algo2} are well defined and
stay in the feasible set $\mathcal{K}$ almost surely, and this immediately leads to the desired conclusion that the above INGD algorithm can find the $(\delta,\epsilon)$-stationary point for the $\mathcal{H}_\infty$ state-feedback synthesis problem with the same high-probability iteration complexity $\mathcal{O}\left(\frac{\Delta L^2}{\epsilon^3 \delta}\log(\frac{\Delta}{p\delta\epsilon})\right)$. 
Now we formalize Theorem \ref{thm5} as follows. 

\begin{thm} \label{thm: finite-time bounds}
Consider the $\mathcal{H}_\infty$ state-feedback control problem \eqref{eq:hinfopt} with the objective function $J(K)$ defined in \eqref{eq:hinfcost}. Suppose Assumption \ref{assump1} holds, and the initial policy is stabilizing, i.e. $K^0\in \mathcal{K}$. Denote $\Delta_0 := \dist(\mathcal{K}_{J(K^0)},\mathcal{K}^c)$. Let $\Delta=J(K^0) - J^*$. Denote the Lipschitz constant of $J(K)$ over the sublevel set $\mathcal{K}_{J(K^0)}$ as $L_0$.  For any  $\delta < \Delta_0 $, the iterations generated by the INGD algorithm will stay in $\mathcal{K}$ almost surely. In addition, the INGD method  finds  a $(\delta,\epsilon)$-stationary point with a high-probability iteration complexity $\mathcal{O}\left(\frac{\Delta L_0^2}{\epsilon^3 \delta}\log(\frac{\Delta}{p\delta \epsilon})\right)$, where $p$ is the failure probability.
\end{thm}
\begin{proof}
We first use an induction argument to show that the iterates of INGD are well defined and stay inside the feasible set $\mathcal{K}$ (almost surely) when $\delta < \Delta_0$. For $n=0$, notice that the variable $\Xi$ (used in Algorithm \ref{algo3}) has to be in $\mathcal{K}$ almost surely and well defined.  The reason is that we sample $\Xi$ uniformly from $[K,K - \delta^0 \frac{\Upsilon}{\|\Upsilon \|_2}]$ and $\delta^0 < \Delta_0$. Next, the update direction $F^0$ satisfies either $\| F^0 \|_2 \le \epsilon$ or $J(K^0 - \delta^0  F^0 /\| F^0 \|_2) \le J(K^0) -  \frac{\delta^0}{4}\|F^0\|_2$.  If $\|F^0\| \le \epsilon$, the INGD algorithm will terminate as we have found a $(\delta,\epsilon)$-stationary point. If $\|F^0\| > \epsilon$, we have $K^{1} = K^0 - \delta^0 F^0/\|F^0\|_2$. We must have $K^1 \in \mathcal{K}$ as well since $\delta^0 < \Delta_0$. In addition, since $J(K^1) \le J(K^0)$, we have $K^1 \in \mathcal{K}_{J(K^0)}\subset \mathcal{K}$. Then we can repeat this argument to show $K^n \in \mathcal{K}_{J(K^0)}$ for all $n$ almost surely, and the high probability complexity bound becomes just a direct consequence of \cite[Theorem 2.6]{davis2021gradient}.  
\end{proof}

\section{Numerical Experiments}
In this section, we provide simulation results to support our theory. The simulations are executed on a desktop computer with a 3.3 GHz Intel Xeon W-1350 processor, and the implementation is done using MATLAB R2021b. 
We tested GS, NS, model-free NS, and INGD on several examples, and the details are given below.

\subsection{More details on the simulation results in Figure \ref{Simulationrelts}}

\textbf{Implementation of INGD, GS, NS, and model-free NS methods:}
For the problem matrices given in \eqref{set_matric}, we have $J^* = 7.3475$.  Now we discuss how we obtain the left plot in Figure \ref{Simulationrelts}.
In order to obtain the Clark stationary point by using INGD method, we need to  run the INGD algorithm with $\delta$ decreasing to 0. To this end, we start with $\epsilon = 1\times 10^{-5}$ and $\delta = 0.01$. Whenever the INGD method successfully finds a $(\delta,\epsilon)$-stationary point, we reduce $\delta$ by a constant factor $0.7$.  As $\delta$ decreases to 0, one should expect that the INGD algorithm approaches a Clark stationary point. In addition, we also implement the GS and NS methods, where the algorithm parameters are set as in Table \ref{GSpara}. Finally, to implement the model-free NS method for the given problem matrices \eqref{set_matric}, we choose the sample size in the $\mathcal{H}_\infty$-norm estimation oracle as $N = 100$ (this just means that we approximate the LTI system over a 100-step time window). 

\begin{table}
  \caption{Algorithm parameters for GS and NS}
  \label{GSpara}
  \centering
  \begin{tabular}{llllllllllll}
    \toprule
    Parameter  &$(n_x,n_u,m)$    & $\delta^0$ &$\epsilon^0$ &$\mu_\delta$ &$\mu_\epsilon$ &$\delta_{opt}$ &$\epsilon_{opt}$ &$\beta$ &$\theta$ &$\alpha_n$   \\
    \midrule
    GS &$(3,1,4)$   & 0.01 &100 &0.5 &0.5 &0 &0 &0.5 &0.9  &N/A     \\
    NS &$(3,1,4)$   & 0.01 &100 &0.5 &0.5 &0 &0 &0.5 &0.9  &${0.1}/(n+1)$      \\
    \bottomrule
  \end{tabular}
\end{table}
 
\textbf{Implementation of randomly generated cases}: The middle plot of Figure \ref{Simulationrelts} demonstrates the performance of the NS algorithm on some randomly generated cases. In particular, we set $A\in \mathbb{R}^{3\times 3}$ to be $I + \xi$,  where each element of $\xi \in \mathbb{R}^{3 \times 3}$ is sampled uniformly from $[0,1]$. We set $B\in \mathbb{R}^{3\times 1}$ with each element uniformly sampled from $[0,1]$. We set $Q = I + \zeta I \in \mathbb{R}^{3\times 3}$ with $\zeta$ uniformly sampled from $[0,0.1]$. We set $R \in \mathbb{R}$ uniformly sampled from $[1,1.5]$. For each experiment, the initial condition $K^0 \in \mathbb{R}^{1\times 3}$ is also randomly sampled such that $\rho(A-BK^0) < 1$.

\textbf{Implementation of Model-free NS}: The right plot of Figure \ref{Simulationrelts} provides the simulation results of the model-free NS method with different choices of sample size $N$ used in the $\mathcal{H}_\infty$-norm estimation oracle.
The hyperparameters of the model-free NS method are set to be the same as the ones used in NS. The only new issue is that we need to specify a sample size in the MIMO power iteration method for evaluating $J(K)$ from sampled trajectories.
Specifically, we choose $N=[100\,\, 50\,\, 20\,\, 10]$.  
As we can see from the right plot in Figure \ref{Simulationrelts}, as we decrease the sample size, the $\mathcal{H}_\infty$-norm oracle become more noisy, and the algorithm performance has been degraded with more oscillations.

\textbf{Gradient oracle for differentiable points.}
To implement INGD and GS, we need an oracle which is capable of computing the gradient of $J(K)$ for any differentiable points. We provide two options for implementing such gradient oracle. 
For the first option, 
we can just use a finite-difference scheme to estimate the gradient at any differentiable points. 
Specifically, for any differentiable point $K$, we can estimate $\nabla J(K)$ as a matrix whose $(i,j)$-th entry is computed as $\frac{J(K + h V_{ij})-J(K - h V_{ij})}{2h}$, where $V_{ij}\in \R^{n_u\times n_x}$ is a matrix whose $(i,j)$-th entry is 1 and all other entries are 0, and $h$ is a small positive parameter (e.g. $h = 0.0001$).
For the second option, we can compute the gradient at differentiable points using some explicit analytic formula based on singular value decomposition. Specifically, we can use a special case of the subgradient analytical formula provided in \cite{apkarian2006nonsmooth}.
To this end, let us rewrite $J(K)$ as follows:
\begin{align} \label{hin_cost2}
J(K) =\sup_{\omega\in[0, 2\pi]}\sigma_{\max} (H(K,\omega) ),
\end{align}
where $H(K,\omega) = (Q+K^{\tp}RK)^{\frac{1}{2}}(e^{j\omega}I-A+BK)^{-1}$ and $\sigma_{\max}$ denotes the maximum singular value. Notice that $J(K)$ is differentiable at $K$ if the $\mathcal{H}_\infty$ norm of $H(K,\omega)$ is achieved at one frequency $\omega_0$ and the largest singular value of $H(K,\omega_0)$ has multiplicity one \cite{Subramani93}. 
Then we can perform the singular value decomposition of $H(K,\omega_0)$ and take $u_1$ (which is the unit left-singular vector) and $v_1$ (which is the unit right-singular vector) corresponding to the largest singular value.  
Denote $H_1(K) = (Q+K^\tp R K)^{\frac{1}{2}}$ , $H_2(K) = (e^{j\omega_0}I - A+BK)^{-1}$, and define $\Gamma = \int_0^\infty  e^{-\tau H_1(K)} H_2(K) v_1 u_1^* e^{-\tau H_1(K)} d\tau$. Then  the gradient of $J(K)$ at differentiable point $K$ can be calculated as $\nabla J(K) =  \operatorname{Re}(R K (\Gamma + \Gamma^\tp) -(H_2(K) v_1 u_1^* H(K,\omega_0) B)^\tp )$. The derivation of the above formula uses the chain rule of the total derivative, and can be viewed as a special case of the general subgradient formula given in \cite{apkarian2006nonsmooth}. 
The two options lead to similar performances, as shown in 
Figure \ref{fig:twooracles}.  We tested both options for implementing INGD and GS with the problem matrices given in  \eqref{set_matric}. From Figure \ref{fig:twooracles},  we can see that both options work well and generate  similar trajectories.

\begin{figure}
\minipage{0.45\textwidth}
  \includegraphics[width=\linewidth]{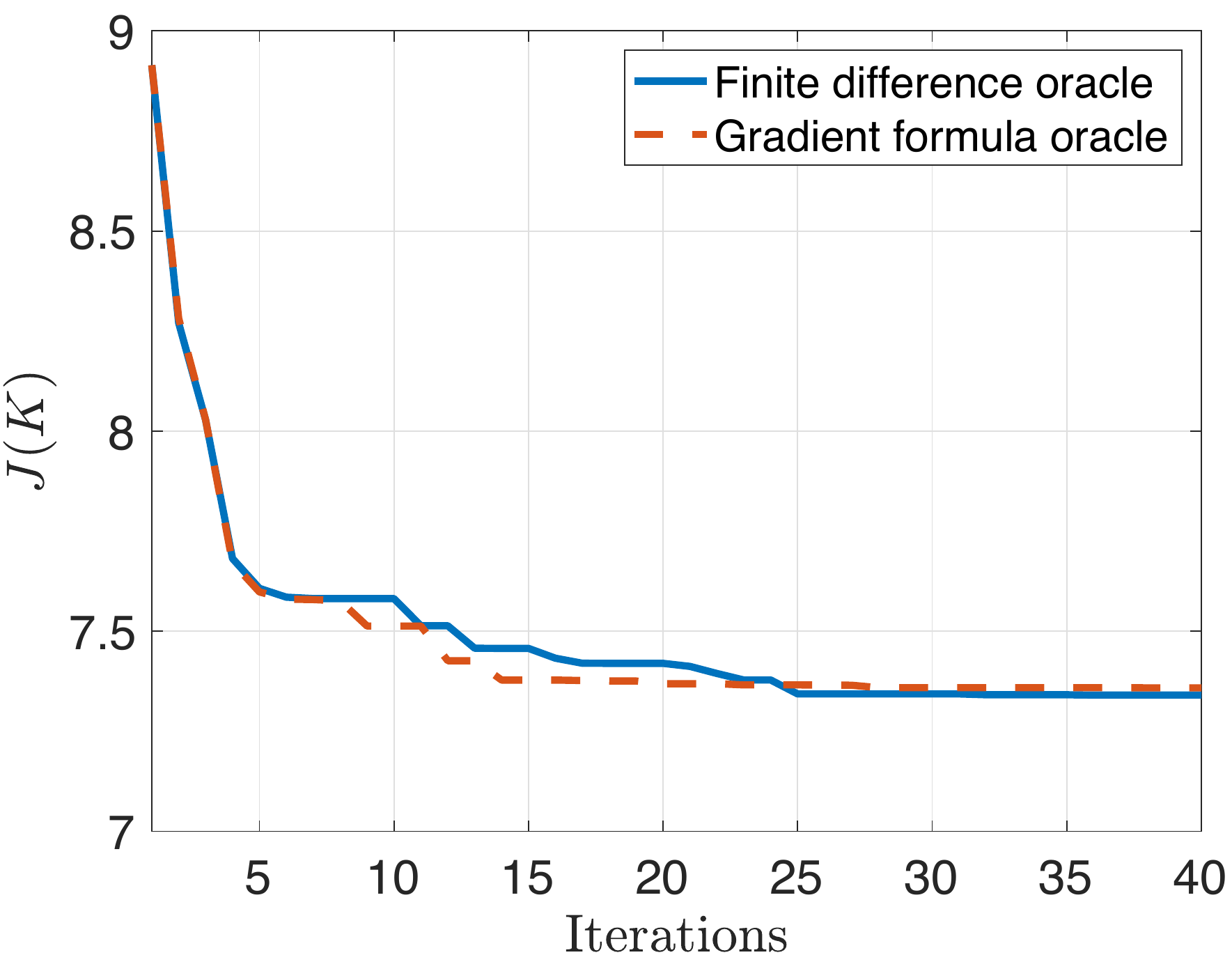}
\endminipage\hfill
\minipage{0.45\textwidth}
  \includegraphics[width=\linewidth]{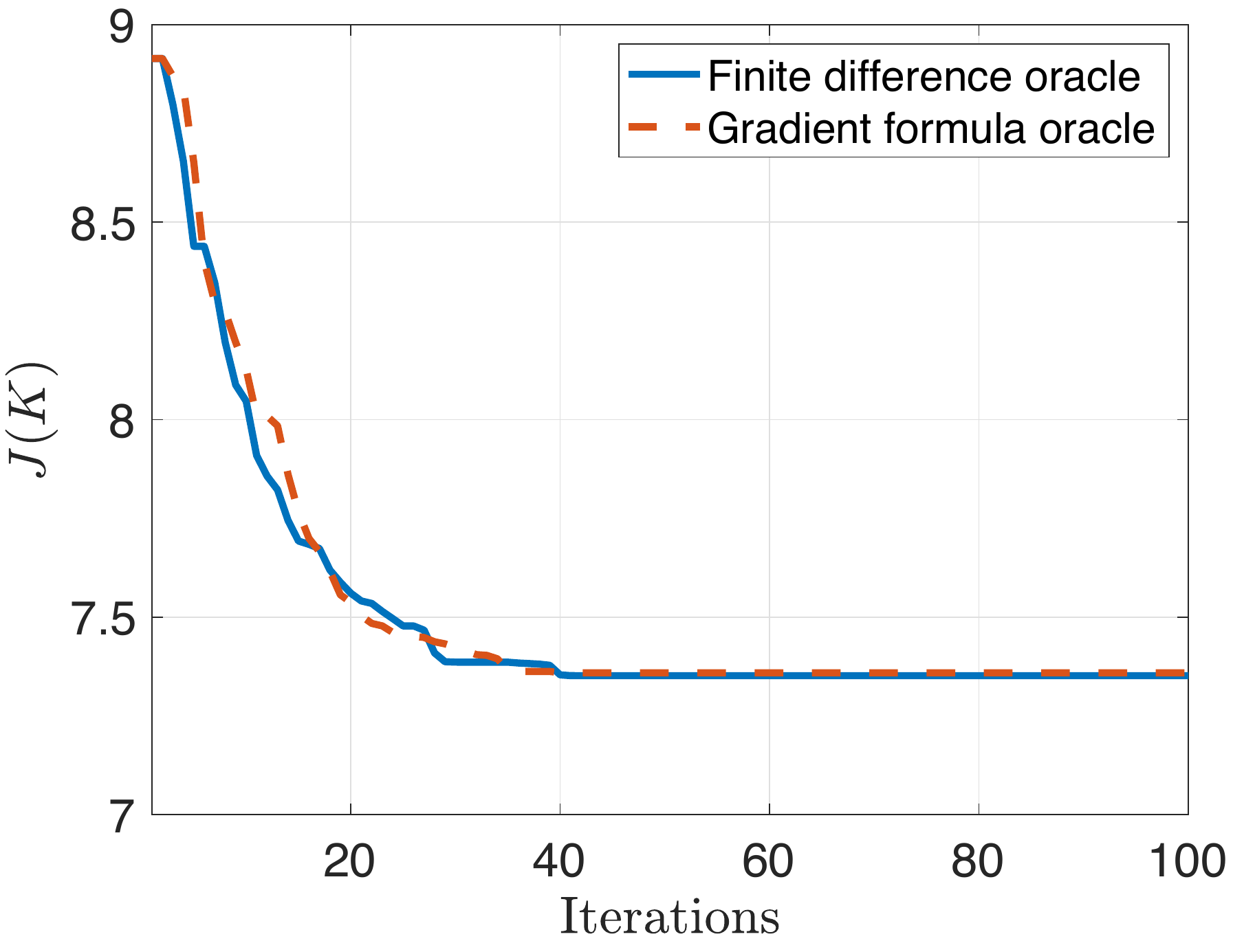}
\endminipage
 \caption{Simulation result for finite-difference oracle and gradient formula oracle. Left: INGD method. Right: GS method.} \label{fig:twooracles}
\end{figure}

\subsection{Numerical results for INGD with constant choices of $(\delta,\epsilon)$}

The simulation results for INGD with constant choices of $(\delta, \epsilon)$ are shown in Figure \ref{fig:INGD}, where we set $\delta = 0.01$ and $\epsilon = 1\times 10^{-8}$.
From the left figure of Figure \ref{fig:INGD}, it can be seen that it takes 10 steps for INGD to find a $(\delta,\epsilon)$-stationary point. At step 10, we have $\| F \|_2 < \epsilon$. However, $J(K)$ does not converge to the optimal value as shown in the right plot of Figure \ref{fig:INGD}. This  result  confirms that $(\delta,\epsilon)$-stationarity does not imply being $\delta$-close to an
$\epsilon$-stationary point of $J$ as commented in  \cite{pmlr-v119-zhang20p,davis2021gradient}.

\begin{figure}
\minipage{0.45\textwidth}
  \includegraphics[width=\linewidth]{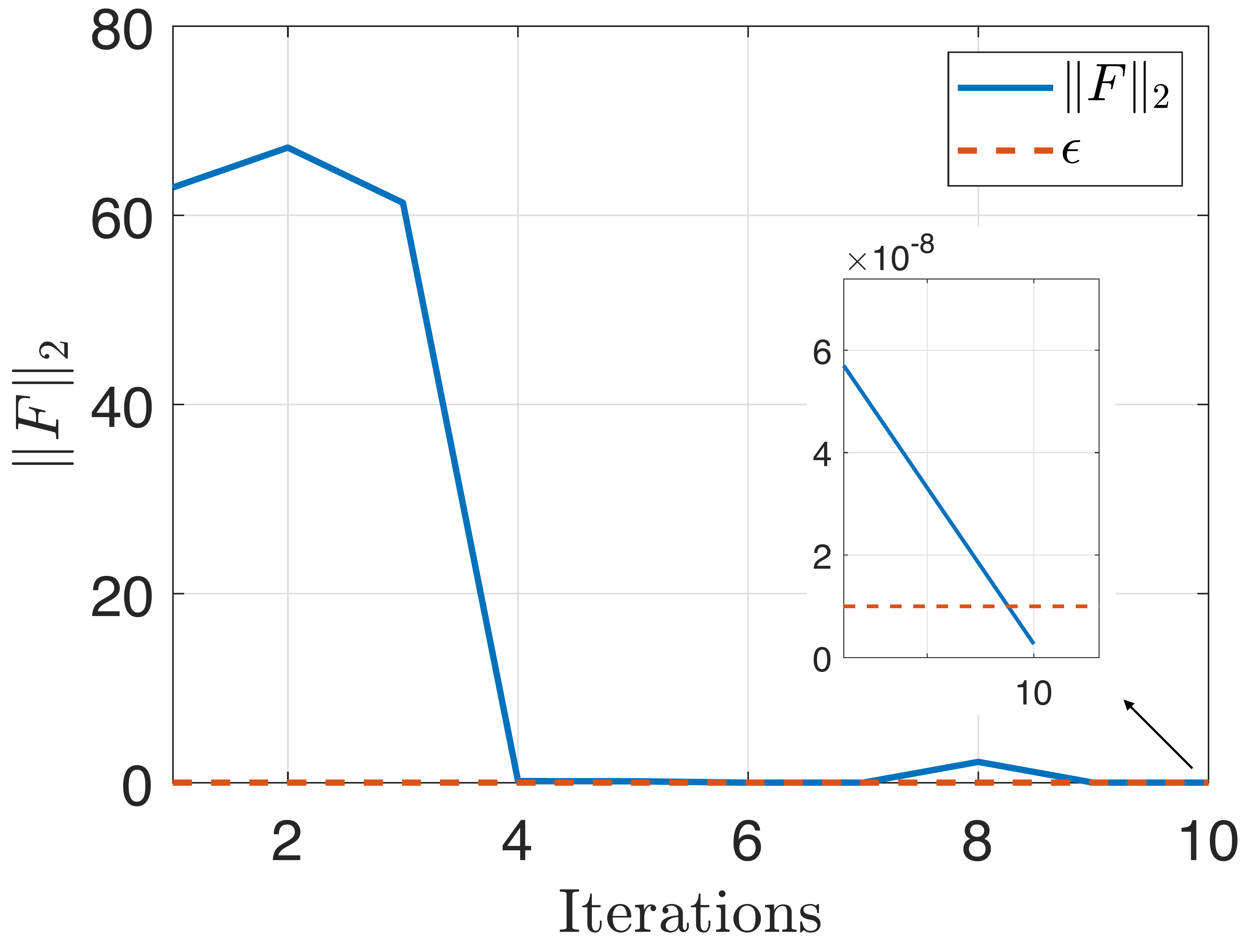}
\endminipage\hfill
\minipage{0.45\textwidth}
  \includegraphics[width=\linewidth]{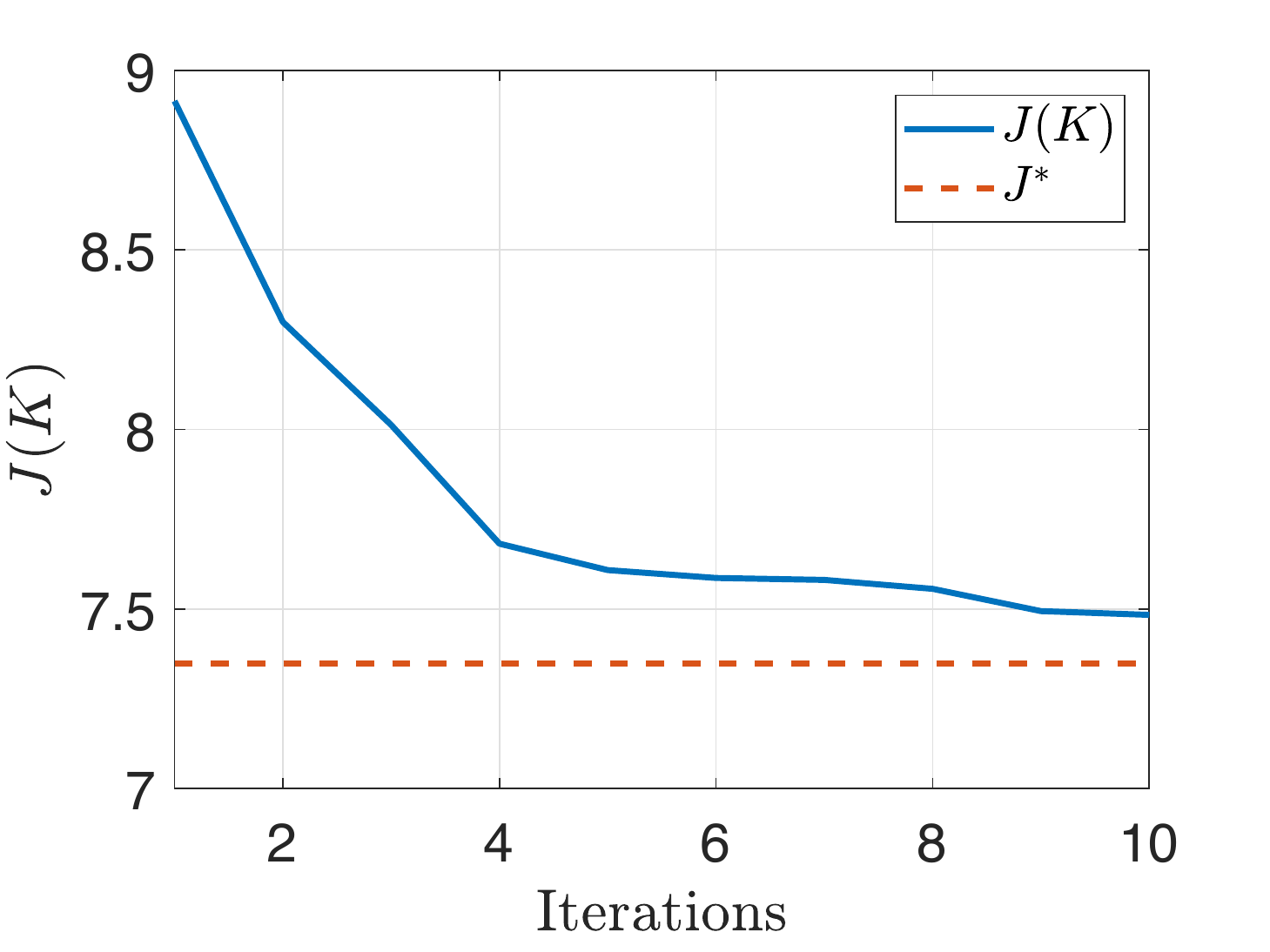}
\endminipage
 \caption{Simulation result for one-time INGD methods with $\delta = 0.01$ and $\epsilon = 1\times 10^{-8}$. Left: the trajectory of the $\|F \|_2$. Right: the trajectory of function value $J(K)$.} \label{fig:INGD}
\end{figure}

\subsection{More examples}
We further tested the NS and GS methods on some slightly larger examples. In particular, 
we first consider the following set of problem matrices:
\begin{align} 
        &A = \bmat{1.7865 &0.3912 &0.8758 &0.5996\\
    0.2756 &1.3175 &0.7692 &0.4848 \\
    0.4764 &0.9786 &1.0618 &0.7591 \\
    0.4489 &0.7918 &0.6014 &1.7520}, \,\, 
B = \bmat{0.1303 &0.0312\\ 0.1309 &0.0528 \\ 0.7452 &0.6727\\ 0.2460 &0.0743}, \label{set_matric2} \\ 
& Q = 1.0613I_{4}, \,\,
R = 1.1315I_{2}. \nonumber
\end{align} 
For the above example, we have $J^* = 43.26$ and we select the initialization $K^0$ to be:
$$K^0 = \bmat{2.4364  &2.2337 &2.4867   &1.5551 \\    12.1213   &-4.6823    &2.1718   &-2.5906},$$
which satisfies $\rho(A - BK^0) = 0.9567 < 1$.
The results are reported in the left plot of Figure \ref{hig_dim}.

In addition, we also perform the NS and GS algorithm on some randomly generated cases. Similarly, we set $A\in \mathbb{R}^{4\times 4}$ to be $I + \xi$,  where each element of $\xi \in \mathbb{R}^{4 \times 4}$ is sampled uniformly from $[0,1]$. We set $B\in \mathbb{R}^{4\times 2}$ with each element uniformly sampled from $[0,1]$. We set $Q = I + \zeta I \in \mathbb{R}^{4\times 4}$ with $\zeta$ uniformly sampled from $[0,0.1]$, and $R = I+ \upsilon I \in \mathbb{R}^{2\times 2}$ with $\upsilon$ uniformly sampled from $[0,0.5]$. For each experiment, the initial condition $K^0 \in \mathbb{R}^{2\times 4}$ is also chosen such that $\rho(A-BK^0) < 1$.

Figure \ref{hig_dim} shows the simulation results on these examples. The left plot demonstrates the convergence of NS and GS methods with the problem matrices in \eqref{set_matric2}. The middle  and right plots demonstrate the performance of NS and GS on the randomly generated examples, respectively. It can be seen that both GS and NS work quite well on these examples. 

\begin{figure}
\minipage{0.33\textwidth}
  \includegraphics[width=\linewidth]{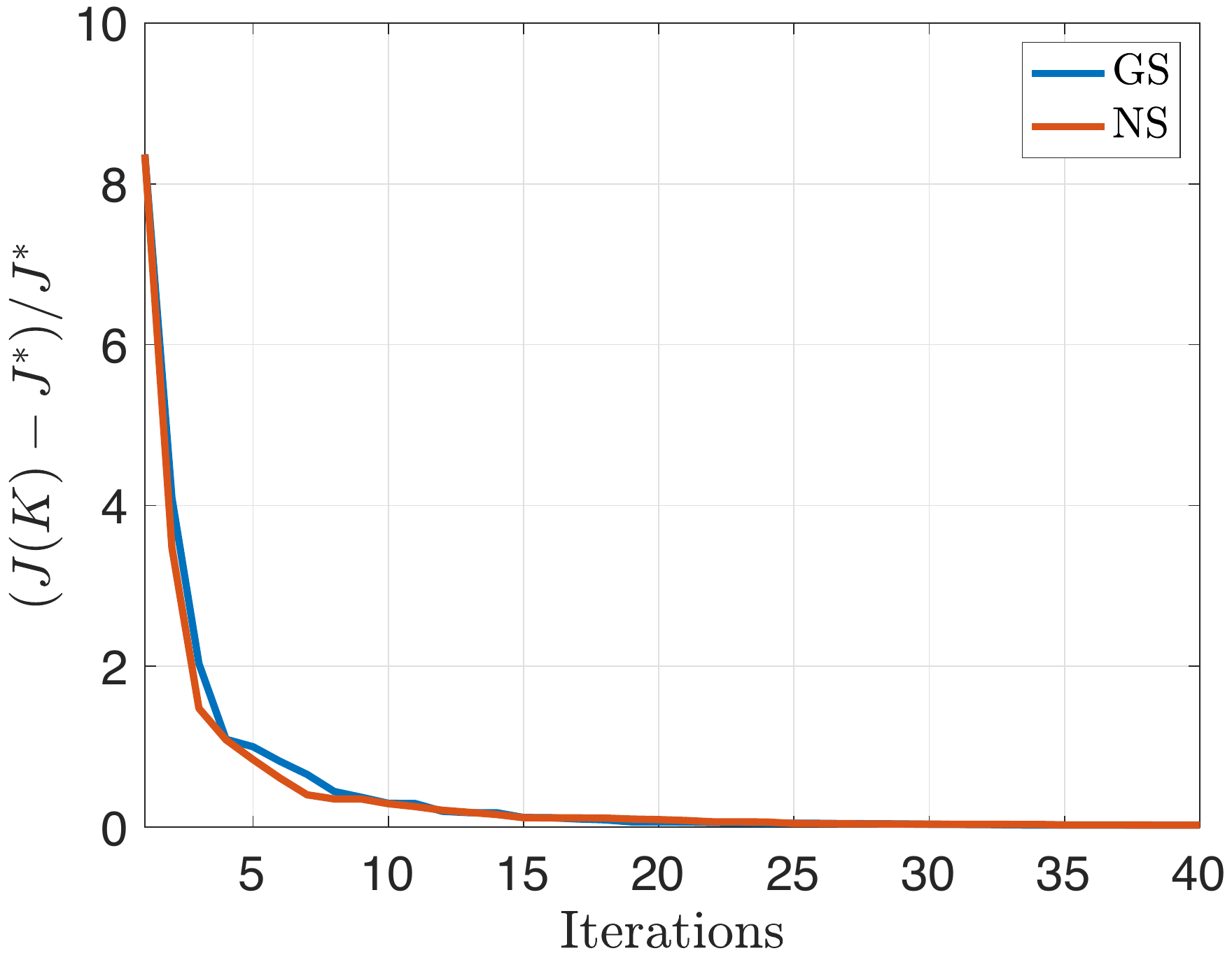}
\endminipage\hfill
\minipage{0.33\textwidth}%
  \includegraphics[width=\linewidth]{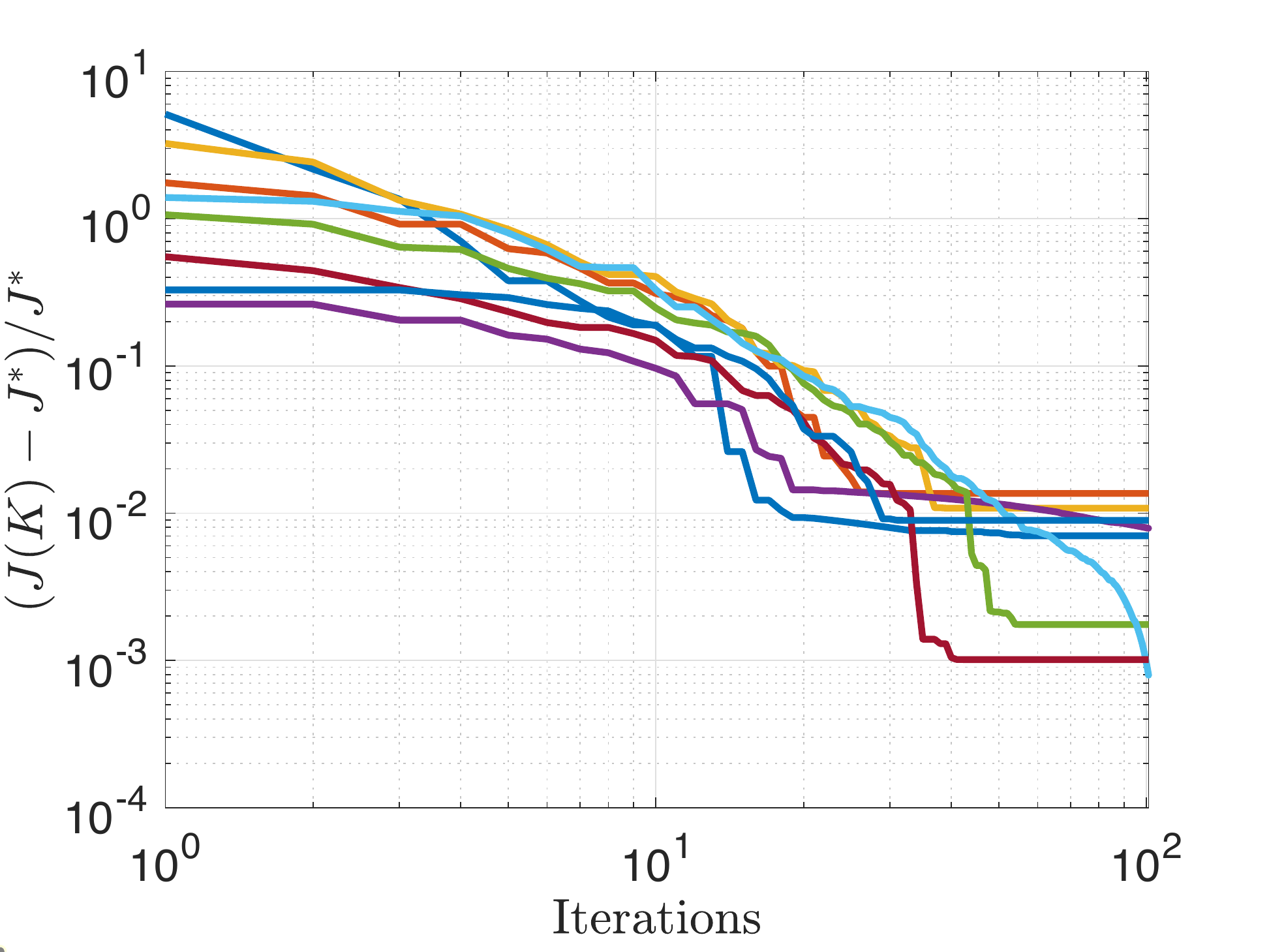}
\endminipage
\minipage{0.33\textwidth}%
  \includegraphics[width=\linewidth]{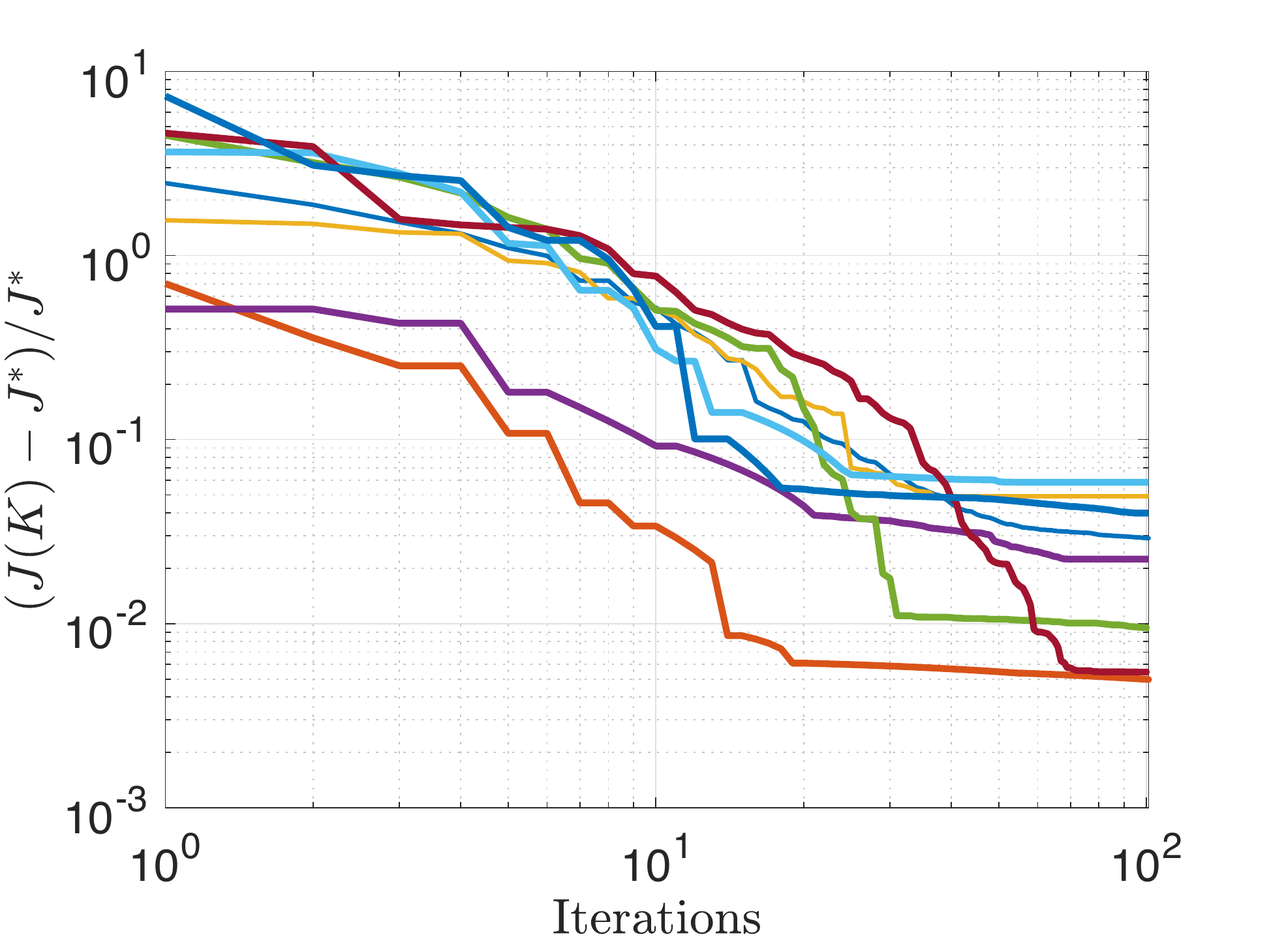}
\endminipage
 \caption{Simulation results for higher dimensional cases. Left: The trajectory of relative error of NS with problem matrices \eqref{set_matric2}. Middle: The trajectory of relative error of 8 randomly generated cases for NS method. Right: The trajectory of relative error of 8 randomly generated cases for GS method.} \label{hig_dim}
\end{figure}

\section{Further Discussions}

\subsection{Uniqueness of the minimizing set}

Our theory does not answer whether the global minimum for the $\mathcal{H}_\infty$ state-feedback synthesis problem is unique or not. As commented in Section \ref{sec:land}, the path-connectedness of $\mathcal{K}_\gamma$ for every $\gamma$ can be used to show that there exists a unique global minimizing set in a certain sense \cite[Sections 2\&3]{martin1982connected}. 
However, we are not able to rule out the possibility that the uniuqe global minimizing set actually consists of multiple points. Whether the global minimum of the $\mathcal{H}_\infty$ state-feedback control problem is unique or not is an interesting open question. 

\subsection{Possible generalizations for nonlinear systems}

For nonlinear systems, it is possible to generalize
Theorems \ref{thm3} and \ref{thm5}. The proofs for the finite-time complexity of finding $(\delta,\epsilon)$-stationary points can be generalized to the constrained policy optimization setting, as long as the cost function $J$ is coercive over the nonconvex feasible set $\mathcal{K}$. In contrast, the convergence to global minimum may be too much to ask for general nonlinear robust control problems.
If the cost function $J$ is not coercive, one may just add regularization to induce coerciveness such that one can still find some approximated $(\delta, \epsilon)$-stationary points provably. Such developments are worth more investigations in the future.

\end{document}